\documentclass{amsart}
\usepackage{amsmath}
\usepackage{amssymb}
\usepackage[all]{xy}
\usepackage{graphicx}  
\usepackage{mathrsfs}

\numberwithin{equation}{section}

\usepackage{url}

\usepackage[latin1]{inputenc}
\usepackage{xspace,amssymb,amsfonts,euscript}
\usepackage{amsthm,amsmath}
\usepackage{palatino}
\usepackage{euscript}
\input xy \xyoption {all}

\usepackage{tikz,environ}
\usetikzlibrary{patterns,snakes}

\makeatletter
\providecommand{\env@tikzpicture@save@env}{}
\providecommand{\env@tikzpicture@process}{}
\makeatother

\RequirePackage{color}
\definecolor{myred}{rgb}{0.75,0,0}
\definecolor{mygreen}{rgb}{0,0.5,0}
\definecolor{myblue}{rgb}{0,0,0.65}

\RequirePackage{ifpdf}
\ifpdf
  \IfFileExists{pdfsync.sty}{\RequirePackage{pdfsync}}{}
  \RequirePackage[pdftex,
   colorlinks = true,
   urlcolor = myblue, 
   citecolor = mygreen, 
   linkcolor = myred, 
   pagebackref,
   bookmarksopen=true]{hyperref}
\else
  \RequirePackage[hypertex]{hyperref}
\fi

\RequirePackage{ae, aecompl, aeguill} 

\usepackage[capitalise]{cleveref}

  \def\hg{{\mathfrak h}}

    \def\RM{{\mathbb{R}}}

    \def\ZM{{\mathbb{Z}}}

    \def\FC{{\mathcal{F}}}

  \def\yb{{\mathbf y}}

\def\HS{{\EuScript H}}

\def\a{\alpha}
\def\b{\beta}
\def\g{\gamma}

\def\d{\delta}

\def\e{\varepsilon}

\def\l{\lambda}

\def\r{\rho}
\def\s{\sigma}

\def\z{\zeta}

\newcommand{\nc}{\newcommand} \newcommand{\renc}{\renewcommand}

\newcommand{\rdots}{\mathinner{ \mkern1mu\raise1pt\hbox{.}
    \mkern2mu\raise4pt\hbox{.}
    \mkern2mu\raise7pt\vbox{\kern7pt\hbox{.}}\mkern1mu}}

\def\un{\underline}

\def\to{\rightarrow}

\def\longto{\longrightarrow}

\nc{\triright}{\stackrel{[1]}{\to}}
\nc{\longtriright}{\stackrel{[1]}{\longto}}

\nc{\Hb}{H^\bullet}

\nc{\Br}{\mathcal{B}}
\nc{\HotRR}{{}_R\mathcal{K}_R}
\nc{\HotR}{\mathcal{K}_R}
\nc{\excise}[1]{}

\nc{\h}[1]{\underline{H}_{#1}}

\nc{\Ga}{\mathbb{G}_a} 
\nc{\Gm}{\mathbb{G}_m} 

\nc{\Perv}{{\mathbf{P}}}

\nc{\IH}{{\mathrm{IH}}}

\nc{\ic}{\mathbf{IC}}

\nc{\gl}{{\mathfrak{gl}}}
\renc{\sl}{{\mathfrak{sl}}}
\renc{\sp}{{\mathfrak{sp}}}

\renc{\Im}{\textrm{Im}}

\nc{\HBM}{H^{BM}}

 \DeclareMathOperator{\Hom}{Hom}
 
\DeclareMathOperator{\End}{End} 

\DeclareMathOperator{\id}{id}

\newtheorem{thm}{Theorem}[section]
\newtheorem{lem}[thm]{Lemma}

\newtheorem{prop}[thm]{Proposition}
\newtheorem{cor}[thm]{Corollary}
\newtheorem{claim}[thm]{Claim}

\theoremstyle{definition}
\newtheorem{defi}[thm]{Definition}

\theoremstyle{remark}
\newtheorem{remark}[thm]{Remark}

\nc{\simto}{\stackrel{\sim}{\to}}

\nc{\SD}{\mathbb S \mathcal D}
\DeclareMathOperator{\defect}{df}
\DeclareMathOperator{\Kar}{Kar}

\title{Soergel calculus and Schubert calculus}

\author{Xuhua He}
\address{Department of Mathematics, University of Maryland, College Park, MD 20742, USA and department of Mathematics, HKUST, Hong Kong}
\email{xuhuahe@math.umd.edu}

\author{Geordie Williamson}
\address{Max-Planck-Institut f\"ur Mathematik, Vivatsgasse 7, 53111,
  Bonn, Germany.}
\email{geordie@mpim-bonn.mpg.de}

\begin{document}

\begin{abstract} We reduce some key calculations of compositions of morphisms between
  Soergel bimodules (``Soergel calculus'')
  to calculations in the nil Hecke ring (``Schubert calculus''). This
  formula has several applications in modular representation
  theory.
\end{abstract}

\maketitle

\section{Introduction}

Determining the modular irreducible characters of the symmetric groups and of
simple groups of Lie types is a major open problem in representation
theory. Recently it has been discovered that
there are certain hidden monoidal categories lurking ``behind'' or
``inside'' these categories of
representations. The two most
prominent examples of such categories are the
Khovanov-Lauda-Rouquier categorifications of modules over Kac-Moody Lie
algebras and module categories for Soergel bimodules (which
categorify modules for Weyl groups or their associated Hecke algebras).

In this paper we concentrate on the monoidal category of Soergel
bimodules. Although the definition of the category of
Soergel bimodules is elementary (as a full subcategory of bimodules
over a polynomial ring), calculations can be prohibitively
difficult. Performing such calculations is important, as they often
provide a means of calculating an irreducible character or a
decomposition number. (This is a theorem for rational
representations of algebraic groups via results of
Soergel \cite{So00} and Fiebig \cite{Fie11}. A conjecture of Riche and the second
author \cite{RW} would imply that these calculations also gives all
characters of tilting modules for algebraic groups, and hence all
decomposition numbers for symmetric groups.)

Progress on the problem of performing calculations with Soergel
bimodules has been made by Libedinsky \cite{Li08}, Elias-Khovanov \cite{EK}, Elias
\cite{EliasCathedral} and
Elias-Williamson \cite{EW2} which culminates in a presentation of the monoidal
category of Soergel bimodules by generators and relations. The
description is diagrammatic and has led to progress on
understanding Soergel bimodules and its module categories.
Though much simpler than calculations in
bimodules, these diagrammatic calculations can still be very difficult.

In particular, it is desirable to find additional simplifications that
make calculations more feasible. In this paper we take the first step
in this direction: we show how to calculate a subset of certain constants governing
the behaviour  of Soergel bimodules (entries in ``intersection
forms'') by a simple formula in the nil Hecke ring (see Theorem \ref{thm:main}). This is a significant
simplification, and we believe our formula will have other
applications. In the final section of this paper we show that our
formula can be used to easily rederive interesting examples
discovered by Kashiwara-Saito and Braden. In \cite{W2} the
second author uses this formula to construct many more such examples, and
deduces that the exceptional characteristics occurring in Lusztig's
conjectured character formula grow exponentially in the rank.

The authors have the optimistic hope that one should be able to reduce
all essential calculations amongst Soergel bimodules to calculations
in the nil Hecke ring. That this is in principle possible is evidenced
by the work of Dyer. Our goal is explicit formulas that
allow us to calculate the characters of the indecomposable Soergel
bimodules algorithmically. This paper can be seen as a first step in
this direction.

\subsection{Acknowledgements} We would like to thank
Ben Elias, Thorge Jensen, Nicolas Libedinsky and Britta Sp\"ath for useful remarks.

\section{Background}

\subsection{} Let $S$ be a finite set and $(m_{s t})_{s, t \in S}$ be a matrix with entries in $\mathbb N \cup \{\infty\}$ such that $m_{s s}=1$ and $m_{s t}=m_{t s} \ge 2$ for all $s \neq t$. Let $W$ be a group generated by $S$ with relations $(s t)^{m_{s t}}=1$ for $s, t \in S$ with $m_{s, t}< \infty$. We say that $(W, S)$ is a {\it Coxeter system} and $W$ a {\it Coxeter group}. The Coxeter group $W$ is equipped with the length function $\ell: W \to \mathbb N$ and the Bruhat order $\le$. 

An {\it expression} is a finite sequence of elements of $S$. We denote by $Ex(S)=\sqcup_{i \in \mathbb N} S^{ i}$ the set of all expressions of $(W, S)$. Let $\un{w}=(s_1, s_2, \cdots, s_m) \in Ex(S)$. The length $\ell(\un{w})$ of $\un{w}$ is $m$. A {\it subexpression} of $\un{w}$ is a sequence $(s_1^{e_1}, s_2^{e_2}, \cdots, s_m^{e_m})$, where $e_i \in \{0, 1\}$ for all $i$. We call $\un{e}$ the associated {\it $01$-sequence} and simply write $(s_1^{e_1}, s_2^{e_2}, \cdots, s_m^{e_m})$ as $\un{w}^{\un{e}}$. 

The group multiplication gives a natural map 
\[
Ex(S) \to W, \qquad \un{w} \mapsto \un{w}_{\bullet},
\]
where $\un{w}_{\bullet}=s_1 s_2 \cdots s_m$ for $\un{w}=(s_1, s_2,
\cdots, s_m)$. Notice that $\ell(\un{w}_{\bullet}) \le
\ell(\un{w})$. If equality holds, then we call $\un{w}$ a {\it reduced
  expression}.

Given a subexpression $\un{w}^{\un{e}}$ of $\un{w}$ 
we set $(\un{w}^{\un{e}})_\bullet = s_1^{e_1}s_2^{e_2} \dots s_m^{e_m}$.

\subsection{} Now we recall the Demazure product.  \label{demprod}

Let $x, y \in W$. By \cite[Lemma 1]{He08}, the set $\{u v; u \le x, v \le y\}$ contains a unique maximal element. We denote this element by $x*y$ and call it the {\it Demazure product} of $x$ and $y$. Then 
\[
\{u v; u \le x, v \le y\}=\{w \in W; w \le x*y\}.
\]

In particular, if $x' \le x$ and $y' \le y$, then $(x')*(y') \le x*y$. 

The operator $*$ gives a monoidal structure on $W$. This monoidal structure gives another natural map 
\[
Ex(S) \to W, \qquad \un{w} \mapsto \un{w}_*,
\]
where $\un{w}_*=s_1*s_2*\cdots*s_m$ for $\un{w}=(s_1, s_2, \cdots, s_m)$. 

By definition, $(\un{w}^{\un{e}})_\bullet \le \un{w}_*$ for any $01$-sequence $\un{e}$. In particular $\un{w}_{\bullet} \le \un{w}_*$. 

\subsection{} Given an expression $\un{w}=(s_1, s_2, \cdots, s_m)$, a $01$-sequence $\un{e}=(e_1, e_2, \cdots, e_m)$ and $0 \le k \le m$, 
we set $\un{w}_{\le k}=(s_1, \cdots, s_k)$, $\un{e}_{\le k} =(e_1, \cdots, e_k)$ and $w_k=(\un{w}_{\le k}^{\un{e}_{\le k}})_{\bullet}$. 

We define the decorated sequence $(d_1 e_1, d_2 e_2, \cdots, d_m e_m)$ associated to $(\un{w}, \un{e})$. Here $d_i \in \{U, D\}$ is the decoration to $e_i$ (U stands for ``Up'' and D stands for ``Down''). The decoration is defined as follows. For any $i$, 
\[
d_i = \begin{cases} U &\text{if $w_{i-1}s_i > w_{i-1}$,} \\
D &\text{if $w_{i-1}s_i < w_{i-1}$.} \end{cases}
\]
It will often be convenient to view $\un{e}$ as the string
\[(d_1e_1,d_2e_2, \dots, d_me_m)\] in the symbols $U0, U1, D0, D1$. As the decoration is
determined by $\un{e}$ and $\un{w}$, this is no more information. The
\emph{defect} of $\un{e}$ is defined to be
\[ \defect(\un{e}) := \sharp\{i; d_i e_i=U0\}-\sharp\{i; d_i e_i=D0\}.\]

\subsection{} The Hecke algebra $\HS$ of $W$ is the free $\mathbb Z[v, v^{-1}]$-algebra with basis $H_w$ for $w \in W$ and multiplication given by 
\begin{gather*}
H_x H_y=H_{x y}, \qquad \text{ if } \ell(x y)=\ell(x)+\ell(y); \\
(H_s+v) (H_s-v^{-1})=0, \qquad \text{ for } s \in S. 
\end{gather*}

For $\un{w}=(s_1, \cdots, s_m) \in Ex(S)$, we set $H_{\un{w}}=H_{s_1}
\cdots H_{s_m}$. Then $H_{\un{w}}=H_{\un{w}_\bullet}$ if $\un{w}$ is a
reduced expression. 

The Hecke algebra $\HS$ has a $\mathbb Z$-linear {\it bar involution}
$\bar \, : \HS \to \HS$ sending $v$ to $v^{-1}$ and $H_w$ to
$H_{w^{-1}}^{-1}$ for all $w \in W$. For any $w \in W$, the
Kazhdan-Lusztig element $\un{H}_w$ is the unique bar-invariant
(i.e. $\overline{h} = h$) element in $H_w+\sum_{x<w} v \mathbb Z[v] H_x$. The elements $\{\un{H}_w \mid w \in W\}$ form a basis of $\HS$. 

For $\un{w}=(s_1, \cdots, s_m) \in Ex(S)$, we set $\un{H}_{\un{w}}=\un{H}_{s_1}
\cdots \un{H}_{s_m}$. 


\subsection{} \label{real}
We fix a realization (in the sense of \cite[\S 3.1]{EW2}) $\hg$
of $W$ over a commutative ring $\Bbbk$. Recall that this consists of a
free and finitely generated $\Bbbk$-module $\hg$ together with subsets
\[
\{ \a_s \}_{s \in S} \subset \hg^* \quad \text{and} \quad \{ \a^\vee_s \}_{s \in S} \subset \hg
\]
of ``roots'' and ``coroots'' such that $\langle \alpha_s,
\alpha_s^\vee \rangle = 2$ for all $s \in S$ and 
  the formulas (for $s \in S$ and
$v \in \hg$)
\[
s(v) := v - \langle \alpha_s , v \rangle \alpha^\vee_s
\]
define an action of $W$ on $\hg$.

For simplicity in this paper we will assume that one of the following
two assumptions is satisfied:
\begin{enumerate}
\item $\Bbbk = \RM$ and $\hg$ is the geometric representation of $W$
  (defined for example in \cite[\S5.3]{Hu});
\item $\hg$ is obtained by extension of scalars from a realization
  defined over $\ZM$ for which the matrix $(\langle \alpha_s^\vee,
  \alpha_t \rangle)_{s, t \in S}$ is a generalized Cartan matrix (in
  the sense of \cite[Chapters 1 \& 3]{Ka}).
\end{enumerate}
Additionally, we assume that the maps
$\alpha_s : \hg \to \Bbbk$ and $\alpha_s^\vee : \hg^* \to \Bbbk$ are
surjective. (This condition is called \emph{Demazure surjectivity} in
\cite{EW2}. It is automatic if $2$ is invertible in $\Bbbk$.)
 
\begin{remark}
  It is possible (and can be interesting, see \cite{EliasSatake}) to consider more
  general realizations, however this can introduce extra
  subtleties. For example one may not have a good notion of positive
  roots and Demazure operators
  need not satisfy the braid relations. It is for this reason that
  we make the assumptions above.
\end{remark}

We denote by $R =S(\hg^*)$ the symmetric algebra of $\hg^*$ over
$\Bbbk$. We view $R$ as a graded $\Bbbk$-algebra with $\deg \hg^*  =
2$. Because $W$ acts on $\hg^*$, it also acts on $R$ by functoriality.
For any $s \in S$ we define the \emph{Demazure operator} $\partial_s :
R \to R[-2]$ by
\[
\partial_s(f) = \frac{f - sf}{\alpha_s}.
\]
(This is well defined under our assumptions on $\hg$, see \cite[\S
3.3]{EW2}.)

\subsection{} \label{subsec:diags}
An {\it $S$-graph} is a finite, decorated, planar
graph with boundary properly embedded in the planar strip $\mathbb R \times [0,1]$
whose edges are colored by $S$ and all of
whose vertices are of the following types:
\begin{enumerate}
\item univalent vertices (``dots''):
 $ \begin{array}{c}
\tikz[scale=0.5]{\draw[dashed] (3,0) circle (1cm);
\draw[color=blue] (3,-1) to (3,0);
\node[circle,fill,draw,inner sep=0mm,minimum size=1mm,color=blue] at (3,0) {};}
\end{array}$
\item trivalent vertices:
  $\begin{array}{c}
\tikz[scale=0.5]{\draw[dashed] (0,0) circle (1cm);
\draw[color=blue] (-30:1cm) -- (0,0) -- (90:1cm);
\draw[color=blue] (-150:1cm) -- (0,0);}
\end{array}$
\item $2m_{st}$-valent vertices:
$\begin{array}{c}\tikz[scale=0.5,rotate=-11]{\draw[dashed] (0,0) circle (1cm);
\draw[color=blue] (0,0) -- (22.5:1cm);
\draw[color=blue] (0,0) -- (67.5:1cm);
\draw[color=blue] (0,0) -- (112.5:1cm);
\draw[color=blue] (0,0) -- (157.5:1cm);
\draw[color=blue] (0,0) -- (-22.5:1cm);
\draw[color=blue] (0,0) -- (-67.5:1cm);
\draw[color=blue] (0,0) -- (-112.5:1cm);
\draw[color=blue] (0,0) -- (-157.5:1cm);
\draw[color=red] (0,0) -- (0:1cm);
\draw[color=red] (0,0) -- (45:1cm);
\draw[color=red] (0,0) -- (90:1cm);
\draw[color=red] (0,0) -- (135:1cm);
\draw[color=red] (0,0) -- (180:1cm);
\draw[color=red] (0,0) -- (-45:1cm);
\draw[color=red] (0,0) -- (-90:1cm);
\draw[color=red] (0,0) -- (-135:1cm);
}\end{array}$
\\
(We require that there are exactly
$2m_{st}$ edges originating from the vertex, and that they are
alternately colored $s$ and $t$ around the vertex. For example, the
pictured example has $m_{{\color{red}s}{\color{blue}t}} = 8$.)
\end{enumerate}
Additionally any
$S$-graph may have its regions (the connected components of the complement
of the graph in $\mathbb R \times [0,1]$) decorated by boxes
containing homogenous elements of $R$.

Here is an example of an $S$-graph (with $m_{{\color{blue}s},
  {\color{red}t}} = 5$, $m_{{\color{blue}s},
  {\color{green}u}} = 2$, $m_{{\color{green}u},
  {\color{red}t}} = 3$):
\[
\begin{tikzpicture}[scale=0.8]
  \coordinate (b) at (2.7,1);
  \coordinate (a) at (2,2);
\coordinate (c) at (3.7,2.5);
  \coordinate (d) at (0.5,3);
\coordinate (dd) at (-0,2);
  \coordinate (e) at (3,3.5);
  \coordinate (ed) at (3.5,3);
\coordinate (f) at (4.8,3.4);
\coordinate (z0) at (0.6,2.3);
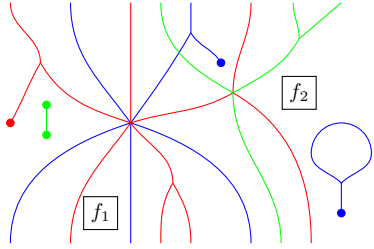
\coordinate (z1) at (0.6,1.8);

\node[rectangle,draw,scale=0.8] at (1.5,0.5) {$f_1$};
\node[rectangle,draw,scale=0.8] at (4.8,2.5) {$f_2$};

\draw[blue] (0,0) to[out=90,in=-162] (a);\draw[blue] (2,0) to (a);
\draw[blue] (4,0) to[out=90,in=-18] (a);  
\draw[blue] (a) to[out=126,in=-90] (1,4);\draw[blue] (a) to[out=54,in=-120] (e);
\draw[blue] (e) to (3,4);
\draw[blue] (ed) to[out=120, in=-60] (e);
 \node[circle,fill,draw,inner
      sep=0mm,minimum size=1mm,color=blue] at (ed) {};

\draw[red] (1,0) to [out=90,in=-126] (a) to[out=162,in=-60] (d) to[out=90,in=-90] (0,4);
\draw[red] (2.5,0) to[out=90,in=-120] (b) to[out=90,in=-54] (a) to[out=90,in=-90] (2,4);
\draw[red] (3,0) to[out=90,in=-60] (b);
\draw[red] (a) to[out=18,in=-150] (c) to[out=90,in=-90] (4,4);
\draw[red] (5,0) to[out=90,in=-30] (c);
\draw[red] (dd) to[out=60, in=-120] (d);
 \node[circle,fill,draw,inner
      sep=0mm,minimum size=1mm,color=red] at (dd) {};

\draw[green] (4.5,0) to[out=90,in=-90] (c) to[out=150,in=-90] (2.5,4);
\draw[green] (c) to[out=30,in=-120] (f) to[out=90,in=-90] (4.7,4);
\draw[green] (f) to (5.5,4);
\draw[green] (z0) to (z1);
 \node[circle,fill,draw,inner
      sep=0mm,minimum size=1mm,color=green] at (z0) {};
 \node[circle,fill,draw,inner
      sep=0mm,minimum size=1mm,color=green] at (z1) {};
\draw[blue] (5.5,1) to[out=135,in=-90] (5,1.5) to[out=90,in=180] (5.5,2) to[out=0,in=90] (6,1.5)
to[out=-90,in=45] (5.5,1) to[out=-90,in=90] (5.5,0.5);
 \node[circle,fill,draw,inner
      sep=0mm,minimum size=1mm,color=blue] at (5.5,0.5) {};
\end{tikzpicture}
\]
where $f_i \in R$ are homogenous polynomials.
We define the \emph{degree} of an $S$-graph as the sum over the degrees of
its vertices and boxes, where each box has degree equal to the degree of the
corresponding element of $R$, and the vertices have degrees given by
the following rule: dots have degree 1, trivalent vertices have degree
-1 and $2m_{st}$-valent vertices have degree 0. For example, the degree
of the $S$-graph above is $-1 + 1 -1 + 1 -1 -1 + 1 + 1+ \deg f_2 + \deg f_2 =
\deg f_1 + \deg f_2$.

 The boundary points of any $S$-graph on $\mathbb R \times\{0\}$
 and on $\mathbb R \times \{1\}$ gives two sequences of colored points, and
hence two elements in $Ex(S)$. We call these two sequences the
\emph{bottom boundary} and \emph{top boundary}. For example, the
bottom (resp. top) boundary of the $S$-graph above is $(s,t,s,t,t,s,u,t)$
(resp. $(t,s,t,u,s,t,u,u)$).

\subsection{} We now define the diagrammatic category of Soergel
bimodules. Much greater detail and generality can be found in
\cite{EliasCathedral, EW2}. The important case of $W$ of type $A$ is
discussed in detail in \cite{EK}. Our intention is to give the reader a
 summary.

 Let $\SD$ be the monoidal category whose objects are
$\un{w} \in Ex(S)$. For any $\un{x}, \un{y} \in Ex(S)$,
$\Hom_{\SD}(\un{x}, \un{y})$ is defined to be the free $R$-module
generated by isotopy classes of $S$-graphs with bottom boundary $\un{x}$ and top boundary
$\un{y}$, modulo the local relations below. The structure of a
monoidal category is induced by horizontal and vertical concatenation
of diagrams.

Here are the relations. We use the coloring ${\color{red}s}$ and
{\color{blue}$t$}.

\subsubsection{Frobenius unit.}\label{2.7.1}
\begin{gather*}
  \begin{array}{c}
    \tikz[scale=0.7]{\draw[dashed] (0,0) circle (1cm);
\draw[color=red] (-90:1cm) -- (0,0) -- (90:1cm);
\draw[color=red] (-0:0.5cm) -- (0,0);
\node[circle,fill,draw,inner sep=0mm,minimum size=1mm,color=red] at (-0:0.5cm) {};
}
  \end{array}
=
  \begin{array}{c}
    \tikz[scale=0.7]{\draw[dashed] (0,0) circle (1cm);
\draw[color=red] (-90:1cm) -- (0,0) -- (90:1cm);}
  \end{array}.
\end{gather*}
 
\subsubsection{Frobenius associativity.}\label{2.7.2}
\begin{gather*}
  \begin{array}{c}
    \tikz[scale=0.7]{\draw[dashed] (0,0) circle (1cm);
\draw[color=red] (-45:1cm) -- (0.3,0) -- (45:1cm);
\draw[color=red] (135:1cm) -- (-0.3,0) -- (-135:1cm);
\draw[color=red] (-0.3,0) -- (0.3,0);
}
  \end{array}
=
  \begin{array}{c}
    \tikz[scale=0.7,rotate=90]{\draw[dashed] (0,0) circle (1cm);
\draw[color=red] (-45:1cm) -- (0.3,0) -- (45:1cm);
\draw[color=red] (135:1cm) -- (-0.3,0) -- (-135:1cm);
\draw[color=red] (-0.3,0) -- (0.3,0);
}
  \end{array}.
\end{gather*}

\subsubsection{Needle relation.}
\begin{gather*}
  \begin{array}{c}
    \begin{tikzpicture}[scale=0.7]
      \draw[dashed] (0,0) circle (1cm);
      \draw[red] (0,0) circle (0.6cm);
      \draw[red] (0,0.6) --(0,1);
\draw[red] (0,-0.6) --(0,-1);
    \end{tikzpicture}
  \end{array}= 0.
\end{gather*}
\subsubsection{Barbell relation.}
\begin{gather*}
  \begin{array}{c}
    \tikz[scale=0.7]{\draw[dashed] (0,0) circle (1cm);
\draw[color=red] (0,-0.6) -- (0,0.6);
\node[circle,fill,draw,inner sep=0mm,minimum size=1mm,color=red] at (0,0.6) {};
\node[circle,fill,draw,inner sep=0mm,minimum size=1mm,color=red] at (0,-0.6) {};
}
  \end{array}
=
  \begin{array}{c}
    \tikz[scale=0.7,rotate=90]{\draw[dashed] (0,0) circle (1cm);
\draw (-0.5,-0.5) rectangle (0.5,0.5);
\node at (0,0) {$\alpha_{\color{red} s}$};
}
  \end{array}.
\end{gather*}

\subsubsection{Nil Hecke relation.} \label{eq:nilHecke}
\begin{gather*}
  \begin{array}{c}
    \begin{tikzpicture}[scale=0.7]
      \draw[dashed] (0,0) circle (1cm);
      \draw[red] (0,1) --(0,-1);
\draw (0.2,-0.4) rectangle (0.7,0.4);
\node at (0.4,0) {$f$};
    \end{tikzpicture}
  \end{array}= 
  \begin{array}{c}
    \begin{tikzpicture}[scale=0.7]
      \draw[dashed] (0,0) circle (1cm);
      \draw[red] (0,1) --(0,-1);
\node at (-0.4,0) {${\color{red} s}f$};
\draw (-0.1,-0.4) rectangle (-0.7,0.4);
    \end{tikzpicture}
  \end{array}
+ 
  \begin{array}{c}
    \begin{tikzpicture}[scale=0.7]
      \draw[dashed] (0,0) circle (1cm);
      \draw[red] (0,1) --(0,0.6);\node[circle,fill,draw,inner
      sep=0mm,minimum size=1mm,color=red] at (0,0.6) {}; 
      \draw[red] (0,-1) --(0,-0.6);\node[circle,fill,draw,inner
      sep=0mm,minimum size=1mm,color=red] at (0,-0.6) {}; 
\node at (0,0) {$\partial_{\color{red} s}f$};
\draw (-0.5,-0.4) rectangle (0.5,0.4);
    \end{tikzpicture}
  \end{array}
\end{gather*}
(See \S \ref{real} for the definition of $\partial_s$.)

\subsubsection{Two-color associativity.}\label{2.7.6} We give the first three cases i.e.
$m_{{\color{red}s}{\color{blue}t}} =2, 3, 4$. The reader
can probably guess the general form (see
\cite{EliasCathedral} (6.12) for all the details).

$m_{st} = 2$ (type $A_1 \times A_1$):
\begin{gather*}
  \begin{array}{c}
    \begin{tikzpicture}
      \draw[dashed] (0,0) circle (1cm);
\draw[red] (-45:1) -- (135:1);
\draw[blue] (-135:1) -- (45:0.3);
\draw[blue] (45:0.3) -- (20:1);
\draw[blue] (45:0.3) -- (80:1); 
    \end{tikzpicture}
  \end{array}
=
  \begin{array}{c}
    \begin{tikzpicture}
      \draw[dashed] (0,0) circle (1cm);
\draw[red] (-45:1) -- (135:1);
\draw[blue] (-135:1) -- (-135:0.4);
\draw[blue] (-135:0.4) -- (20:1);
\draw[blue] (-135:0.4) -- (80:1); 
    \end{tikzpicture}
  \end{array}
\end{gather*}

$m_{st} = 3$ (type $A_2$):
\[ \begin{array}{c}
  \begin{tikzpicture}
      \draw[dashed] (0,0) circle (1cm);
\coordinate (a1) at (30:1);\coordinate (a2) at (60:1);\coordinate (a3) at (90:1);
\coordinate (a4) at (135:1);\coordinate (a5) at (-135:1);\coordinate
(a6) at (-90:1); \coordinate (a7) at (-45:1);

\coordinate (l1) at (45:0.4); \coordinate (l0) at (0:0);
\draw[red] (a5) -- (l0) -- (a3);
\draw[red] (l0) -- (a7);

\draw[blue] (a1) -- (l1) -- (a2);
\draw[blue] (l1) -- (l0) -- (a4);
\draw[blue] (l0) -- (a6);
  \end{tikzpicture}
\end{array}
= 
\begin{array}{c}
  \begin{tikzpicture}
      \draw[dashed] (0,0) circle (1cm);
\coordinate (a1) at (30:1);\coordinate (a2) at (60:1);\coordinate (a3) at (90:1);
\coordinate (a4) at (135:1);\coordinate (a5) at (-135:1);\coordinate
(a6) at (-90:1); \coordinate (a7) at (-45:1);

\coordinate (r1) at (120:0.4); \coordinate (r2) at (-60:0.4); \coordinate (r3) at (-135:0.7);

\draw[blue] (a2) to (r1) to (a4);
\draw[blue] (r1) to[out=-90,in=150] (r2) to (a6);
\draw[blue] (r2) to (a1);

\draw[red] (a3) to (r1) to[out=-150,in=90] (r3) to (a5);
\draw[red] (r1) to[out=-30,in=90] (r2) to (a7);
\draw[red] (r2) to[out=-150,in=0] (r3);

  \end{tikzpicture}
\end{array}
\]

$m_{st} = 3$ (type $B_2$):
\[ \begin{array}{c}
  \begin{tikzpicture}
      \draw[dashed] (0,0) circle (1cm);
\coordinate (a0) at (30:1); \coordinate (a1) at (55:1);
\coordinate (a2) at (75:1);\coordinate (a3) at (105:1); \coordinate (a4) at (135:1);
\coordinate (a5) at (-45:1);\coordinate (a6) at (-75:1);
\coordinate (a7) at (-105:1); \coordinate (a8) at (-135:1);

\coordinate (l1) at (40:0.4); \coordinate (l0) at (0:0);

\draw[blue] (a0) -- (l1) -- (a1);
\draw[blue] (l1) -- (l0) -- (a3);
\draw[blue] (a6) -- (l0) -- (a8);

\draw[red] (a5) -- (l0) -- (a7);
\draw[red] (a2) -- (l0) -- (a4);

  \end{tikzpicture}
\end{array}
= 
 \begin{array}{c}
  \begin{tikzpicture}
      \draw[dashed] (0,0) circle (1cm);
\coordinate (a0) at (30:1); \coordinate (a1) at (55:1);
\coordinate (a2) at (75:1);\coordinate (a3) at (105:1); \coordinate (a4) at (135:1);
\coordinate (a8) at (-45:1);\coordinate (a7) at (-75:1);
\coordinate (a6) at (-105:1); \coordinate (a5) at (-135:1);

\coordinate (r1) at (120:0.4); \coordinate (r2) at (-60:0.4); \coordinate (r3) at (-135:0.7);

\draw[blue] (a1) -- (r1) -- (a3);
\draw[blue] (r1) to[out=-135,in=90] (r3) to (a5);
\draw[blue] (r3) to[out=0,in=-135] (r2) to (a7);
\draw[blue] (r1) to[out=-75,in=105] (r2);
\draw[blue] (r2) to (a0);

\draw[red] (a2) -- (r1) -- (a4);
\draw[red] (r1) to[out=-105,in=135] (r2)  to[out=75,in=-45] (r1);
\draw[red] (a6) to (r2) to (a8);

  \end{tikzpicture}
\end{array}
\]

\subsubsection{The Jones-Wenzl relation.}\label{2.7.7} Elias' Jones-Wenzl
relation expresses a dotted $2m_{st}$-vertex
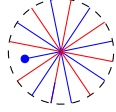
\[
\begin{array}{c}
    \tikz[scale=0.7,rotate=-11]{\draw[dashed] (0,0) circle (1cm);
\draw[color=blue] (0,0) -- (22.5:1cm);
\draw[color=blue] (0,0) -- (67.5:1cm);
\draw[color=blue] (0,0) -- (112.5:1cm);
\draw[color=blue] (0,0) -- (157.5:1cm);
\draw[color=blue] (0,0) -- (-22.5:1cm);
\draw[color=blue] (0,0) -- (-67.5:1cm);
\draw[color=blue] (0,0) -- (-112.5:1cm);
\draw[color=red] (0,0) -- (0:1cm);
\draw[color=red] (0,0) -- (45:1cm);
\draw[color=red] (0,0) -- (90:1cm);
\draw[color=red] (0,0) -- (135:1cm);
\draw[color=red] (0,0) -- (180:1cm);
\draw[color=red] (0,0) -- (-45:1cm);
\draw[color=red] (0,0) -- (-90:1cm);
\draw[color=red] (0,0) -- (-135:1cm);
\draw[color=blue] (0,0) -- (-157.5:0.7);
 \draw[blue] (0,-1) --(0,-0.6);\node[circle,fill,draw,inner
      sep=0mm,minimum size=1mm,color=blue] at (-157.5:0.7) {}; 
}
\end{array}
\]
as a linear combination of diagrams consisting only of dots and
trivalent vertices. We will not give the general form of the relation
here, as the determination of the coefficients is complicated. (To
understand the results of this paper, explicit knowledge of
only the simplest coefficients is necessary.) We give the
details when $m_{{\color{red}s}{\color{blue}t}} = 2,
3, 4$ and refer the reader to \cite{EliasCathedral} for more information.

$m_{st} = 2$ (type $A_1 \times A_1$):
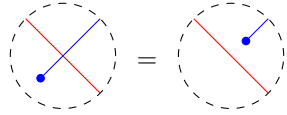
\begin{gather*}
  \begin{array}{c}
    \begin{tikzpicture}[scale=0.7]
      \draw[dashed] (0,0) circle (1cm);
\draw[red] (-45:1) -- (135:1);
\draw[blue] (45:1) -- (-135:0.6);
\node[circle,fill,draw,inner
      sep=0mm,minimum size=1mm,color=blue] at (-135:0.6) {}; 
    \end{tikzpicture}
  \end{array}
=
\begin{array}{c}
    \begin{tikzpicture}[scale=0.7]
      \draw[dashed] (0,0) circle (1cm);
\draw[red] (-45:1) -- (135:1);
\draw[blue] (45:1) -- (-135:-0.4);
\node[circle,fill,draw,inner
      sep=0mm,minimum size=1mm,color=blue] at (-135:-0.4) {}; 
    \end{tikzpicture}
  \end{array}
\end{gather*}

$m_{st} = 3$ (type $A_2$):
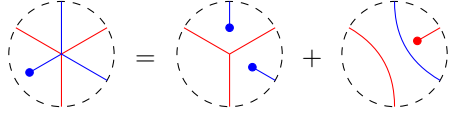
\begin{gather*}
  \begin{array}{c}
    \begin{tikzpicture}[scale=0.7]
      \draw[dashed] (0,0) circle (1cm);
\foreach \r in {30,150,-90}
      \draw[red] (0,0) -- (\r:1cm);
\foreach \r in {90,-30}
      \draw[blue] (0,0) -- (\r:1cm);
\draw[blue] (0,0) -- (-150:0.7);
\node[circle,fill,draw,inner
      sep=0mm,minimum size=1mm,color=blue] at (-150:0.7) {}; 
    \end{tikzpicture}
  \end{array}
=
  \begin{array}{c}
    \begin{tikzpicture}[scale=0.7]
      \draw[dashed] (0,0) circle (1cm);
\foreach \r in {30,150,-90}
      \draw[red] (0,0) -- (\r:1cm);
\foreach \r in {90,-30}
{ \draw[blue] (\r:0.5) -- (\r:1cm);
\node[circle,fill,draw,inner sep=0mm,minimum size=1mm,color=blue] at (\r:0.5) {};}
    \end{tikzpicture}
  \end{array} + 
  \begin{array}{c}
    \begin{tikzpicture}[scale=0.7]
      \draw[dashed] (0,0) circle (1cm);
\draw[red] (-90:1) to[out=90,in=-30] (150:1);
\draw[blue] (90:1) to[out=-90,in=150] (-30:1);
\draw[red] (30:0.5) -- (30:1cm);
\node[circle,fill,draw,inner sep=0mm,minimum size=1mm,color=red] at (30:0.5) {};
    \end{tikzpicture}
  \end{array} 
\end{gather*}

$m_{st} = 4$ (type $B_2$). Here there are two natural choices of
realization, and this affects the coefficients in the Jones-Wenzl
relation.  Assume first that we are using the symmetric ``geometric''
realization, so that 
\[
\langle \alpha_s^\vee, \alpha_t \rangle =
\langle \alpha_t^\vee, \alpha_s \rangle = -\sqrt{2}.\]
Here a Jones-Wenzl relation takes the form
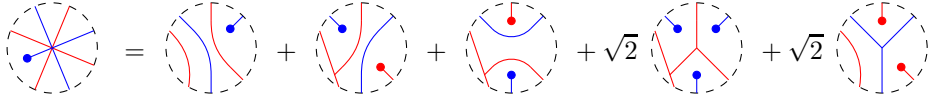
\begin{gather*}
  \begin{array}{c}
    \begin{tikzpicture}[scale=0.6,rotate=-22.5]
      \draw[dashed] (0,0) circle (1cm);
\foreach \r in {0,90,180,-90}
      \draw[red] (0,0) -- (\r:1cm);
\foreach \r in {45,135,-45}
      \draw[blue] (0,0) -- (\r:1cm);
\draw[blue] (0,0) -- (-135:0.6);
\node[circle,fill,draw,inner
      sep=0mm,minimum size=1mm,color=blue] at (-135:0.6) {}; 
    \end{tikzpicture}
  \end{array} = 
  \begin{array}{c}
    \begin{tikzpicture}[scale=0.6]
      \def\r{0.6}
      \draw[dashed] (0,0) circle (1cm);
 \foreach\a in {-90,45,135}
\draw[blue] (\a:1) to (\a:\r);
 \foreach\a in {90,-45}
\draw[red] (\a:1) to (\a:\r);
\draw[red] (-120:1) to[out=80,in=-45] (157.75:1);
\draw[blue] (-90:\r) to[out=90,in=-45] (135:\r);
\draw[red] (90:\r) to[out=-90,in=135] (-45:\r);
\node[circle,fill,draw,inner
      sep=0mm,minimum size=1mm,color=blue] at (45:\r) {}; 
    \end{tikzpicture}
  \end{array}
+ 
  \begin{array}{c}
    \begin{tikzpicture}[scale=0.6]
      \def\r{0.6}
      \draw[dashed] (0,0) circle (1cm);
 \foreach\a in {-90,45,135}
\draw[blue] (\a:1) to (\a:\r);
 \foreach\a in {90,-45}
\draw[red] (\a:1) to (\a:\r);
\draw[red] (-120:1) to (157.75:1);
\draw[red] (-135:\r) to (-135:0.85);
\draw[blue] (-90:\r) to[out=90,in=-135] (45:\r);
\draw[red] (-135:\r) to[out=45,in=-90] (90:\r);
\node[circle,fill,draw,inner
      sep=0mm,minimum size=1mm,color=blue] at (135:\r) {}; 
\node[circle,fill,draw,inner
      sep=0mm,minimum size=1mm,color=red] at (-45:\r) {}; 
    \end{tikzpicture}
  \end{array}
+  
\begin{array}{c}
    \begin{tikzpicture}[scale=0.6]
      \def\r{0.6}
      \draw[dashed] (0,0) circle (1cm);
 \foreach\a in {-90,45,135}
\draw[blue] (\a:1) to (\a:\r);
 \foreach\a in {90,-45}
\draw[red] (\a:1) to (\a:\r);
\draw[red] (-120:1) to (157.75:1);
\draw[red] (-135:\r) to (-135:0.85);
\draw[blue] (135:\r) to[out=-45,in=-135] (45:\r);
\draw[red] (-135:\r) to[out=45,in=135] (-45:\r);
\node[circle,fill,draw,inner
      sep=0mm,minimum size=1mm,color=red] at (90:\r) {}; 
\node[circle,fill,draw,inner
      sep=0mm,minimum size=1mm,color=blue] at (-90:\r) {}; 
    \end{tikzpicture}
  \end{array}
+  
\sqrt{2} \begin{array}{c}
    \begin{tikzpicture}[scale=0.6]
      \def\r{0.6}
      \draw[dashed] (0,0) circle (1cm);
 \foreach\a in {-90,45,135}
\draw[blue] (\a:1) to (\a:\r);
 \foreach\a in {90,-45}
\draw[red] (\a:1) to (\a:\r);
\draw[red] (-120:1) to (157.75:1);
\draw[red] (-135:\r) to (-135:0.85);
\foreach\a in {-135,-45,90}
\draw[red] (\a:\r) to (0,0);
\node[circle,fill,draw,inner
      sep=0mm,minimum size=1mm,color=blue] at (45:\r) {}; 
\node[circle,fill,draw,inner
      sep=0mm,minimum size=1mm,color=blue] at (-90:\r) {}; 
\node[circle,fill,draw,inner
      sep=0mm,minimum size=1mm,color=blue] at (135:\r) {}; 
    \end{tikzpicture}
  \end{array}
+  
\sqrt{2}\begin{array}{c}
    \begin{tikzpicture}[scale=0.6]
      \def\r{0.6}
      \draw[dashed] (0,0) circle (1cm);
 \foreach\a in {-90,45,135}
\draw[blue] (\a:1) to (\a:\r);
 \foreach\a in {90,-45}
\draw[red] (\a:1) to (\a:\r);
\draw[red] (-120:1) to[out=80,in=-45]  (157.75:1);
\foreach\a in {135,45,-90}
\draw[blue] (\a:\r) to (0,0);
\foreach\a in {-45,90}
\node[circle,fill,draw,inner
      sep=0mm,minimum size=1mm,color=red] at (\a:\r) {}; 
    \end{tikzpicture}
  \end{array}
\end{gather*}
and one obtains the other relation by swapping red and
blue.

On the other hand if we take the ``Cartan matrix'' realization with
\[
\langle \alpha_t^\vee, \alpha_s \rangle = -2 \quad \text{and} \quad
\langle \alpha_s^\vee, \alpha_t \rangle = -1
\]
(so that the non-simple positive roots are $\alpha_{\color{red} s} + \a_{\color{blue} t}$ and
$\alpha_{\color{red} s} + 2\a_{\color{blue} t}$) then the relation is
not stable under interchanging red and blue, and the Jones-Wenzl
relations are
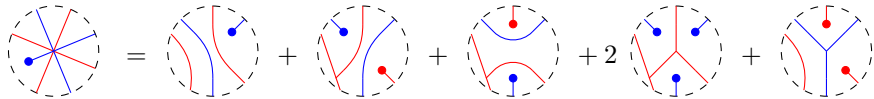
\begin{gather*}
  \begin{array}{c}
    \begin{tikzpicture}[scale=0.6,rotate=-22.5]
      \draw[dashed] (0,0) circle (1cm);
\foreach \r in {0,90,180,-90}
      \draw[red] (0,0) -- (\r:1cm);
\foreach \r in {45,135,-45}
      \draw[blue] (0,0) -- (\r:1cm);
\draw[blue] (0,0) -- (-135:0.6);
\node[circle,fill,draw,inner
      sep=0mm,minimum size=1mm,color=blue] at (-135:0.6) {}; 
    \end{tikzpicture}
  \end{array} = 
  \begin{array}{c}
    \begin{tikzpicture}[scale=0.6]
      \def\r{0.6}
      \draw[dashed] (0,0) circle (1cm);
 \foreach\a in {-90,45,135}
\draw[blue] (\a:1) to (\a:\r);
 \foreach\a in {90,-45}
\draw[red] (\a:1) to (\a:\r);
\draw[red] (-120:1) to[out=80,in=-45] (157.75:1);
\draw[blue] (-90:\r) to[out=90,in=-45] (135:\r);
\draw[red] (90:\r) to[out=-90,in=135] (-45:\r);
\node[circle,fill,draw,inner
      sep=0mm,minimum size=1mm,color=blue] at (45:\r) {}; 
    \end{tikzpicture}
  \end{array}
+ 
  \begin{array}{c}
    \begin{tikzpicture}[scale=0.6]
      \def\r{0.6}
      \draw[dashed] (0,0) circle (1cm);
 \foreach\a in {-90,45,135}
\draw[blue] (\a:1) to (\a:\r);
 \foreach\a in {90,-45}
\draw[red] (\a:1) to (\a:\r);
\draw[red] (-120:1) to (157.75:1);
\draw[red] (-135:\r) to (-135:0.85);
\draw[blue] (-90:\r) to[out=90,in=-135] (45:\r);
\draw[red] (-135:\r) to[out=45,in=-90] (90:\r);
\node[circle,fill,draw,inner
      sep=0mm,minimum size=1mm,color=blue] at (135:\r) {}; 
\node[circle,fill,draw,inner
      sep=0mm,minimum size=1mm,color=red] at (-45:\r) {}; 
    \end{tikzpicture}
  \end{array}
+  
\begin{array}{c}
    \begin{tikzpicture}[scale=0.6]
      \def\r{0.6}
      \draw[dashed] (0,0) circle (1cm);
 \foreach\a in {-90,45,135}
\draw[blue] (\a:1) to (\a:\r);
 \foreach\a in {90,-45}
\draw[red] (\a:1) to (\a:\r);
\draw[red] (-120:1) to (157.75:1);
\draw[red] (-135:\r) to (-135:0.85);
\draw[blue] (135:\r) to[out=-45,in=-135] (45:\r);
\draw[red] (-135:\r) to[out=45,in=135] (-45:\r);
\node[circle,fill,draw,inner
      sep=0mm,minimum size=1mm,color=red] at (90:\r) {}; 
\node[circle,fill,draw,inner
      sep=0mm,minimum size=1mm,color=blue] at (-90:\r) {}; 
    \end{tikzpicture}
  \end{array}
+  
2 \begin{array}{c}
    \begin{tikzpicture}[scale=0.6]
      \def\r{0.6}
      \draw[dashed] (0,0) circle (1cm);
 \foreach\a in {-90,45,135}
\draw[blue] (\a:1) to (\a:\r);
 \foreach\a in {90,-45}
\draw[red] (\a:1) to (\a:\r);
\draw[red] (-120:1) to (157.75:1);
\draw[red] (-135:\r) to (-135:0.85);
\foreach\a in {-135,-45,90}
\draw[red] (\a:\r) to (0,0);
\node[circle,fill,draw,inner
      sep=0mm,minimum size=1mm,color=blue] at (45:\r) {}; 
\node[circle,fill,draw,inner
      sep=0mm,minimum size=1mm,color=blue] at (-90:\r) {}; 
\node[circle,fill,draw,inner
      sep=0mm,minimum size=1mm,color=blue] at (135:\r) {}; 
    \end{tikzpicture}
  \end{array}
+  
\begin{array}{c}
    \begin{tikzpicture}[scale=0.6]
      \def\r{0.6}
      \draw[dashed] (0,0) circle (1cm);
 \foreach\a in {-90,45,135}
\draw[blue] (\a:1) to (\a:\r);
 \foreach\a in {90,-45}
\draw[red] (\a:1) to (\a:\r);
\draw[red] (-120:1) to[out=80,in=-45]  (157.75:1);
\foreach\a in {135,45,-90}
\draw[blue] (\a:\r) to (0,0);
\foreach\a in {-45,90}
\node[circle,fill,draw,inner
      sep=0mm,minimum size=1mm,color=red] at (\a:\r) {}; 
    \end{tikzpicture}
  \end{array}
\end{gather*}

\begin{gather*}
  \begin{array}{c}
    \begin{tikzpicture}[scale=0.6,rotate=-22.5]
      \draw[dashed] (0,0) circle (1cm);
\foreach \r in {0,90,180,-90}
      \draw[blue] (0,0) -- (\r:1cm);
\foreach \r in {45,135,-45}
      \draw[red] (0,0) -- (\r:1cm);
\draw[red] (0,0) -- (-135:0.6);
\node[circle,fill,draw,inner
      sep=0mm,minimum size=1mm,color=red] at (-135:0.6) {}; 
    \end{tikzpicture}
  \end{array} = 
  \begin{array}{c}
    \begin{tikzpicture}[scale=0.6]
      \def\r{0.6}
      \draw[dashed] (0,0) circle (1cm);
 \foreach\a in {-90,45,135}
\draw[red] (\a:1) to (\a:\r);
 \foreach\a in {90,-45}
\draw[blue] (\a:1) to (\a:\r);
\draw[blue] (-120:1) to[out=80,in=-45] (157.75:1);
\draw[red] (-90:\r) to[out=90,in=-45] (135:\r);
\draw[blue] (90:\r) to[out=-90,in=135] (-45:\r);
\node[circle,fill,draw,inner
      sep=0mm,minimum size=1mm,color=red] at (45:\r) {}; 
    \end{tikzpicture}
  \end{array}
+ 
  \begin{array}{c}
    \begin{tikzpicture}[scale=0.6]
      \def\r{0.6}
      \draw[dashed] (0,0) circle (1cm);
 \foreach\a in {-90,45,135}
\draw[red] (\a:1) to (\a:\r);
 \foreach\a in {90,-45}
\draw[blue] (\a:1) to (\a:\r);
\draw[blue] (-120:1) to (157.75:1);
\draw[blue] (-135:\r) to (-135:0.85);
\draw[red] (-90:\r) to[out=90,in=-135] (45:\r);
\draw[blue] (-135:\r) to[out=45,in=-90] (90:\r);
\node[circle,fill,draw,inner
      sep=0mm,minimum size=1mm,color=red] at (135:\r) {}; 
\node[circle,fill,draw,inner
      sep=0mm,minimum size=1mm,color=blue] at (-45:\r) {}; 
    \end{tikzpicture}
  \end{array}
+  
\begin{array}{c}
    \begin{tikzpicture}[scale=0.6]
      \def\r{0.6}
      \draw[dashed] (0,0) circle (1cm);
 \foreach\a in {-90,45,135}
\draw[red] (\a:1) to (\a:\r);
 \foreach\a in {90,-45}
\draw[blue] (\a:1) to (\a:\r);
\draw[blue] (-120:1) to (157.75:1);
\draw[blue] (-135:\r) to (-135:0.85);
\draw[red] (135:\r) to[out=-45,in=-135] (45:\r);
\draw[blue] (-135:\r) to[out=45,in=135] (-45:\r);
\node[circle,fill,draw,inner
      sep=0mm,minimum size=1mm,color=blue] at (90:\r) {}; 
\node[circle,fill,draw,inner
      sep=0mm,minimum size=1mm,color=red] at (-90:\r) {}; 
    \end{tikzpicture}
  \end{array}
+  
 \begin{array}{c}
    \begin{tikzpicture}[scale=0.6]
      \def\r{0.6}
      \draw[dashed] (0,0) circle (1cm);
 \foreach\a in {-90,45,135}
\draw[red] (\a:1) to (\a:\r);
 \foreach\a in {90,-45}
\draw[blue] (\a:1) to (\a:\r);
\draw[blue] (-120:1) to (157.75:1);
\draw[blue] (-135:\r) to (-135:0.85);
\foreach\a in {-135,-45,90}
\draw[blue] (\a:\r) to (0,0);
\node[circle,fill,draw,inner
      sep=0mm,minimum size=1mm,color=red] at (45:\r) {}; 
\node[circle,fill,draw,inner
      sep=0mm,minimum size=1mm,color=red] at (-90:\r) {}; 
\node[circle,fill,draw,inner
      sep=0mm,minimum size=1mm,color=red] at (135:\r) {}; 
    \end{tikzpicture}
  \end{array}
+  
2 \begin{array}{c}
    \begin{tikzpicture}[scale=0.6]
      \def\r{0.6}
      \draw[dashed] (0,0) circle (1cm);
 \foreach\a in {-90,45,135}
\draw[red] (\a:1) to (\a:\r);
 \foreach\a in {90,-45}
\draw[blue] (\a:1) to (\a:\r);
\draw[blue] (-120:1) to[out=80,in=-45]  (157.75:1);
\foreach\a in {135,45,-90}
\draw[red] (\a:\r) to (0,0);
\foreach\a in {-45,90}
\node[circle,fill,draw,inner
      sep=0mm,minimum size=1mm,color=blue] at (\a:\r) {}; 
    \end{tikzpicture}
  \end{array}
\end{gather*}

\begin{remark}
  A useful mnemonic to remember the placement of the coefficients in
  the Jones-Wenzl relation is given by the relation
\begin{gather} \label{eq:alldots}
\begin{array}{c}
    \tikz[scale=0.7,rotate=-11]{\draw[dashed] (0,0) circle (1cm);
\def\l{0.8}
\foreach \a in {22.5,67.5,112.5,157.5,-22.5,-67.5,-112.5,-157.5}
{\draw[color=blue] (0,0) -- (\a:\l);
 \node[circle,fill,draw,inner
      sep=0mm,minimum size=1mm,color=blue] at (\a:\l) {}; }
\foreach \a in {0,45,90,135,180,-45,-90,-135}
{\draw[color=red] (0,0) -- (\a:\l);
 \node[circle,fill,draw,inner
      sep=0mm,minimum size=1mm,color=red] at (\a:\l) {}; }
}
\end{array}
=
  \begin{array}{c}
    \begin{tikzpicture}[scale=0.7]
      \draw[dashed] (0,0) circle (1cm);
\node at (0,0) {$\pi_{st}$};
\draw (-0.5,-0.4) rectangle (0.5,0.4);
    \end{tikzpicture}
  \end{array}
\end{gather}
where $\pi_{st}$ denotes the product of all positive roots in the root
subsystem corresponding to ${\color{red}s}, {\color{blue} t}$. The reader can convince themselves
that each Jones-Wenzl relation given above implies that
\eqref{eq:alldots} holds. With some practice, \eqref{eq:alldots} can
be used to remember the placement of the  coefficients above.
\end{remark}

\subsubsection{Zamolodchikov relations}\label{2.7.8} We will not repeat the
definition of the Zamolodchikov relations here, and instead refer the
reader to \cite[\S 1.4.3]{EW2}. (See \cite{EW3} for 
topological background).

\vspace{0.5cm}

\subsection{} \label{subsec:grading} All relations defining $\SD$ are
homogenous for the grading on $S$-graphs defined in  \S
\ref{subsec:diags}. It follows that the $\SD$ is enriched in graded
$\Bbbk$-modules. Throughout it will be convenient to view $\SD$ as
enriched in graded left $R$-modules via
\[
f \cdot
\begin{array}{c}
\tikz[scale=0.4]{ \draw (-1,-1) rectangle (1,1); \node (a) at (0,0) {$D$};}
\end{array}
= \begin{array}{c}
\tikz[scale=0.4]{ \draw (-1,-1) rectangle (1.3,1); \node (a) at (0.6,0)
  {$D$};
\draw (-0.8,-0.5) rectangle (-0.1,0.5); \node (b) at (-0.45,0) {$f$};
}
\end{array}
\]
for an $S$-graph $D$ and homogenous $f \in R$.

Given a $S$-graph $D$ we denote by $\overline{D}$ the $S$-graph
obtained by flipping the diagram vertically. This operation induces a
contravariant equivalence (also denoted $G \mapsto \overline{G}$ on morphisms) of
the monoidal category $\SD$.

\subsection{} \label{int} 
For any $x \in W$, let $\SD^{\nless x}$ be the quotient category of $\SD$
by the ideal $\SD_{< x}$ of $\SD$ generated by all the morphisms which factor
through $\un{y}$, where $\un{y}$ is a reduced expression for some $y <
x$. Let $\Hom^{\nless x}(-, -)$ denote morphisms in
$\SD^{\nless x}$. The image of
any reduced expression for $x$ in $\SD^{\nless x}$ yields an object
defined up to canonical isomorphism, independent of the choice of
reduced expression (see \cite[6.5]{EW2}). We denote this object
$x$. We have $\End^{\nless x}(x) = R$.

By \cite[Theorem 1.1]{EW2}, $\Hom^{\nless
  x}(\un{w}, x)$ and $\Hom^{\nless x}(x,\un{w})$ are
free (graded) $R$-modules with graded basis given by the (images of the) light leaves
$LL_{\un{w}, \un{e}}$ and $\overline{LL}_{\un{w}, \un{e}}$
respectively, where $\un{e}$ runs over the $01$-sequences with
$(\un{w}^{\un{e}})_\bullet=x$. The light leaves are constructed in
\cite[\S 6.1]{EW2}. For any $\un{w} \in Ex(S)$ the \emph{intersection form} for $\un{w}$ at
$x$ is the $R$-bilinear pairing of graded free $R$-modules  
\[
I_{x,\un{w}} : \Hom^{\nless x}(x,\un{w}) \times \Hom^{\nless x}(\un{w},x)
\to \End^{\nless x}(x) = R.
\]

\subsection{}  \label{ifeg} We give two examples of intersection forms.

First, assume that $W$ is a dihedral group with simple reflections
    ${\color{blue} s}$ and ${\color{red} t}$. Let $\un{w} = sts$, $x =
    s$. 
There are two subexpressions for $x$: ${\bf e}^1 =
    (U1, U0, D0)$ (defect 0) and ${\bf e}^2 = (U0, U0, U1)$ (defect 2). The
    corresponding light leaves morphisms are
\[
l_1 = 
\begin{array}{c}
\begin{tikzpicture}[scale=0.3]
\draw[blue] (0,0) -- (0,1) -- (1,2) -- (1,3);
\draw[blue] (2,0) -- (2,1) -- (1,2) -- (1,3);
\draw[red] (1,0) -- (1,1);
\node[circle,red,draw,fill,minimum size=1mm,inner sep=0pt] (a) at (1,1) {};
\end{tikzpicture}
\end{array}
\qquad \text{and} \qquad
l_2 = 
\begin{array}{c}
\begin{tikzpicture}[scale=0.3]
\draw[blue] (0,0) -- (0,1);
\node[circle,blue,draw,fill,minimum size=1mm,inner sep=0pt] (a) at (0,1) {};
\draw[blue] (2,0) -- (2,1) -- (1,2) -- (1,3);
\draw[red] (1,0) -- (1,1);
\node[circle,red,draw,fill,minimum size=1mm,inner sep=0pt] (a) at (1,1) {};
\end{tikzpicture}
\end{array}.
\]
The intersection form is given by the matrix:
\[ ( l_i \circ \bar{l_j})_{i,j \in \{ 1, 2 \}}=
\left ( \begin{matrix} \langle \alpha_t, \alpha_s^\vee \rangle & \a_t \\ \a_t
  & \a_s\a_t \end{matrix} \right ) \]
For example, the upper left entry follows from the following calculation:
\[
\begin{array}{c}
\begin{tikzpicture}[scale=0.3]
\draw[blue] (1,-3) -- (1,-2) -- (0,-1) -- (0,0) -- (0,1) -- (1,2) -- (1,3);
\draw[blue] (1,-3) -- (1,-2) -- (2,-1) -- (2,0) -- (2,1) -- (1,2) --
(1,3);
\draw[red] (1,-1) -- (1,1);
\node[circle,red,draw,fill,minimum size=1mm,inner sep=0pt] (a) at
(1,1) {};
\node[circle,red,draw,fill,minimum size=1mm,inner sep=0pt] (a) at (1,-1) {};
\end{tikzpicture}
\end{array} = 
\begin{array}{c}
\begin{tikzpicture}[scale=0.3]
\draw[blue] (1,-3) -- (1,-2) -- (0,-1) -- (0,0) -- (0,1) -- (1,2) -- (1,3);
\draw[blue] (1,-3) -- (1,-2) -- (2,-1) -- (2,0) -- (2,1) -- (1,2) --
(1,3);
\node at (1,0) {$\alpha_t$};
\end{tikzpicture}
\end{array}
= 
\begin{array}{c}
\begin{tikzpicture}[scale=0.3]
\draw[blue] (1,-3) -- (1,-2) -- (0,-1);
\draw[blue] (0,1) -- (1,2) -- (1,3);

\node[circle,blue,draw,fill,minimum size=1mm,inner sep=0pt] (a) at
(0,1) {};
\node[circle,blue,draw,fill,minimum size=1mm,inner sep=0pt] (a) at (0,-1) {};

\draw[blue] (1,-3) -- (1,-2) -- (2,-1) -- (2,0) -- (2,1) -- (1,2) --
(1,3);
\node at (-1.5,0) {$\partial_s(\alpha_t)$};
\end{tikzpicture}
\end{array}
+
\begin{array}{c}
\begin{tikzpicture}[scale=0.3]
\draw[blue] (1,-3) -- (1,-2) -- (0,-1) -- (0,0) -- (0,1) -- (1,2) -- (1,3);
\draw[blue] (1,-3) -- (1,-2) -- (2,-1) -- (2,0) -- (2,1) -- (1,2) --
(1,3);
\node at (-1.6,0) {$s(\alpha_t)$};
\end{tikzpicture}
\end{array}
=
\begin{array}{c}
\begin{tikzpicture}[scale=0.3]
\node at (-1.8,0) {$\partial_s(\alpha_t)$};
\draw[blue] (0,-3) -- (0,3);
\end{tikzpicture}
\end{array}
\]
(we use the barbell, nil Hecke and needle relations, as well as
$\partial_s(\alpha_t) = \langle \alpha_t, \alpha_s^\vee \rangle$).
This example shows the existence of torsion in the intersection
cohomology of the Schubert variety indexed by $sts$ if $\langle
\alpha_t, \alpha_s^\vee \rangle < -1$ (as happens in $B_2$, $G_2$ and
$\widetilde{A_1}$).

For the second example, assume that $W$ is of type $D_4$ with generators ${\color{blue}
    s}$, ${\color{black} t}$, ${\color{red} u}$, ${\color{green} v}$
  such that $s$, $u$ and $v$ commute. Let $\un{w} = vuvtsuv$, $x =
  suv$. We only give the part of the intersection form corresponding
  to light leaves morphisms of degree 0. Then there are three subexpressions of defect
  $0$. These subexpressions, and the corresponding light leaves
  maps are the following:




\[
\begin{array}{c}
\begin{tikzpicture}[xscale=-0.42,yscale=-0.42]
\draw[blue] (1,0) -- (1,-1);
\node[circle,blue,draw,fill,minimum size=1mm,inner sep=0pt] at (1,-1)
{};
\draw[blue] (-3,0) -- (-3,-1) -- (-1,-3) -- (-1,-4);

\draw[red] (-2,0) -- (-2,-1) -- (0,-3) -- (0,-4);
\draw[red] (2,0) -- (2,-1) -- (0,-3) -- (0,-4);

\draw[green] (-1,0) -- (-1,-1) -- (1,-3) -- (1,-4);
\draw[green] (3,0) -- (3,-1) -- (1,-3) -- (1,-4);

\draw[black] (0,0) -- (0,-1);
\node[circle,black,draw,fill,minimum size=1mm,inner sep=0pt] at (0,-1)
{};
\end{tikzpicture}\\
\textrm{\tiny U1 U1 U0 U0 D0 D0 U1.}
\end{array} \qquad
\begin{array}{c}
\begin{tikzpicture}[xscale=-0.42,yscale=-0.42]

\draw[blue] (-3,0) -- (-3,-1) -- (-1,-3) -- (-1,-4);
\draw[blue] (1,0) -- (1,-1) -- (-1,-3) -- (-1,-4);

\draw[red] (2,0) -- (2,-1);
\node[circle,red,draw,fill,minimum size=1mm,inner sep=0pt] at (2,-1) {};
\draw[red] (-2,0) -- (-2,-1) -- (0,-3) -- (0,-4);

\draw[green] (-1,0) -- (-1,-1) -- (1,-3) -- (1,-4);
\draw[green] (3,0) -- (3,-1) -- (1,-3) -- (1,-4);

\draw[black] (0,0) -- (0,-1);
\node[circle,black,draw,fill,minimum size=1mm,inner sep=0pt] at (0,-1)
{};
\end{tikzpicture}\\
\textrm{ {\tiny U1 U0 U1 U0 D0 U1 D0}}
\end{array} \qquad
\begin{array}{c}
\begin{tikzpicture}[xscale=-0.42,yscale=-0.42]

\draw[blue] (-3,0) -- (-3,-1) -- (-1,-3) -- (-1,-4);
\draw[blue] (1,0) -- (1,-1) -- (-1,-3) -- (-1,-4);

\draw[red] (-2,0) -- (-2,-1) -- (0,-3) -- (0,-4);
\draw[red] (2,0) -- (2,-1) -- (0,-3) -- (0,-4);

\draw[green] (-1,0) -- (-1,-1) -- (1,-3) -- (1,-4);
\draw[green] (3,0) -- (3,-1);
\node[circle,green,draw,fill,minimum size=1mm,inner sep=0pt] at (3,-1)
{};

\draw[black] (0,0) -- (0,-1);
\node[circle,black,draw,fill,minimum size=1mm,inner sep=0pt] at (0,-1)
{};
\end{tikzpicture}\\
\textrm{\tiny U0 U1 U1 U0 U1 D0 D0}
\end{array}
\]
We leave it to the reader to pair these morphisms and obtain the
intersection form
\[
\left ( \begin{matrix} 0 & -1 & -1 \\ -1 & 0 & -1 \\ -1 & -1 & 0
\end{matrix}   \right )
\]
Note that the determinant of this matrix is -2. This example of
2-torsion in the $D_4$ flag variety was discovered by Braden
\cite[A.18]{W1}.

The goal of this paper is to show that light leaves morphisms
corresponding to sequences without $D1$ (like those above) are
canonical and to give a formula for the corresponding entry in the
intersection form in terms of the nil Hecke ring.

\subsection{} Let $\SD_{\oplus}$ denote the additive graded  
envelope of $\SD$. That is, objects of $\SD_{\oplus}$ are formal
direct sums of shifts of the objects $\un{w}$ in $\SD$, with obvious morphisms. We denote by $\SD_{\oplus,0}$ the
subcategory of $\SD_{\oplus}$ consisting only of degree zero
morphisms, and by $\Kar(\SD)$ the Karoubian envelope of
$\SD_{\oplus,0}$. Then $\Kar(\SD)$  is a $\Bbbk$-linear category with
grading shift  functor $[1]$.

Under the assumptions of \S\ref{real}, the main theorems of
\cite{EW2} give a basis for hom spaces in
$\SD$ and deduce a classification of the indecomposable
objects up to isomorphism and shift: for any $x$ and
reduced expression $\un{x}$ for $x$, $\un{x}$ has a unique summand $b_x$ which
is not a summand up to shift of $\un{y}$ for any $y < x$, and any
indecomposable object in $\Kar(\SD)$ is isomorphic to $b_x[m]$ for
some $x \in W$ and $ m \in \ZM$.

From this one deduces an isomorphism of $\ZM[v,v^{-1}]$-algebras
\[
\e : \HS \simto [\Kar(\SD)] \qquad \qquad \un{H}_{\un{w}} \mapsto \un{w}.
\]
where $[\Kar(\SD)]$ denotes the split Grothendieck group of
$\Kar(\SD)$. For any $\un{w} \in Ex(S)$ one has in $[\Kar(\SD)]$
\[
[\un{w}] = \sum_{x \in W} m_x [b_x]
\]
where $m_x$ denotes the graded rank of the intersection form
$I_{x,\un{w}}$. This explains the central importance of the
intersection forms.

In \cite{EW} it is proved (using Soergel bimodules) that if $\Bbbk$ is
a field of characteristic 0 then $\e$ sends the Kazhdan-Lusztig basis
element $\un{H}_w$ to the class of $b_w$.

\subsection{} Though it will not be used in this paper, we briefly 
explain the connection to Soergel bimodules. Let us assume that
$\Bbbk$ is a field of characteristic $\ne 2$.

For any $s \in S$, we denote by $R^s \subset R$ the invariant subring
and
\[
B_s=R \otimes_{R^s} R[1]
\]
a graded $R$-bimodule. For any $\un{w}=(s_1, \cdots, s_m) \in Ex(S)$, we define the corresponding {\it Bott-Samelson bimodule} to be 
\[
B_{\un{w}}:=B_{s_1} \otimes_R B_{s_2} \otimes_R \cdots \otimes_R B_{s_m}.
\]
Let $\mathbb SBim$ be the category of Soergel bimodules, that is, the
Karoubi envelope of $B_{\un{w}}$ for all $\un{w} \in Ex(S)$ inside the
category of graded $R$-bimodules.

By \cite[Satz]{So07}\footnote{under some mild assumptions on the
  realization $\hg$ (which we ignore here).}
for any $w \in W$, there exists a unique (up
to isomorphism) indecomposable $R$-bimodule $B_w$ which occurs as a
direct summand of $B_{\un{w}}$ for any reduced expression $\un{w}$
with $\un{w}_\bullet=w$, and does not occur as a direct summand of
$B_{\un{x}}$ for any $\un{x}$ with $w \nleq \un{x}_*$. The set
$\{B_w\}_{w \in W}$ is a complete set of indecomposable Soergel
bimodules (up to isomorphism and degree shift).

In \cite{EW2} a functor $\FC : \SD \to \mathbb SBim$ is constructed and it
is proved that $\FC$ induces an equivalence of monoidal categories
$\Kar(\SD) \simto \mathbb SBim$. In particular, $\FC$ maps $b_w$ to
$B_w$, for any $w \in W$.

\section{Gobbling morphisms}

\subsection{} 
In general light leaves morphisms are not canonical, and
this causes many complications. For example it seems
difficult to make all entries in intersection forms canonical.
In this section we introduce gobbling morphisms, which are certain 
canonical morphisms between Soergel bimodules. 

The aim of this section is to prove:

\begin{prop} \label{gobble}
Let $\un{w} = u_1 \dots u_n$ be an expression.
  Any two morphisms $\un{w} \to \un{w}_*$ in $\SD^{\not < \un{w}_*}$ given by diagrams consisting
  only of $2m_{st}$-valent vertices and $\ell(\un{w}) - \ell(\un{w}_*)$
  trivalent vertices are equal.
\end{prop}

\begin{defi}
  We call the morphism $\un{w} \to \un{w}_*$ in $\SD^{\nless \un{w}_*}$ whose unicity is given by
  the previous proposition the \emph{gobbling
    morphism} and denote it by $G_{\un{w}}$.
\end{defi}

\begin{remark}
  One has $\deg G_{\un{w}} = - (\sharp\text{ of trivalent vertices in
  $G_{\un{w}}$}) = \ell(\un{w}_*)- \ell(\un{w}).$
\end{remark}

For $\un{w}$ in the proposition we can consider the subexpression $\un{e}
:= e_1 \dots e_n$ defined inductively by the recipe
\begin{gather} \label{eq:edef}
e_i := \begin{cases} 1 & \text{if $w_{i-1}s_i > w_{i-1}$,} \\
  0 & \text{otherwise}. \end{cases}
\end{gather}
where $w_i := (\un{w}_{\le i}^{\un{e}_{\le i}})_{\bullet}$.
Obviously $(\un{w}^{\un{e}})_\bullet = \un{w}_*$. Also note that the
decoration of $\un{e}$ consists entirely of $U1$'s and $D0$'s. Hence
(any choice of) the light leaf morphism $L_{\un{w}, \un{e}}$
consists only of $2m_{st}$-valent vertices and exactly $\ell(\un{w}) -
\ell(\un{w}_*)$ trivalent vertices. In particular $L_{\un{w}, \un{e}}$
satisfies the conditions of Proposition \ref{gobble}.

Let $l_{\un{w}}$ denote the image of $L_{\un{w}, \un{e}}$ in $D^{\nless
  \un{w}_*}$. We can restate Proposition \ref{gobble} as:

\begin{prop} \label{gobblenew}
  Any morphism $\un{w} \to \un{w}_*$ in $\SD^{\not < \un{w}_*}$ satisfying the
  conditions of Proposition \ref{gobble} is equal to $l_{\un{w}}$.
\end{prop}

This is the version we will prove. We start with some preparatory lemmata. The first lemma shows that the
space in which the gobbling morphism(s) live is free of rank one, and
that this degree is the minimal non-zero degree:

\begin{lem} \label{deg}
  $\Hom^{\not < \un{w}_*}(\un{w}, \un{w}_*)$ is zero in degrees $< \ell(\un{w}_*) -
\ell(\un{w})$ and the degree ${\ell(\un{w}_*) -
  \ell(\un{w})}$ part of $\Hom^{\not < \un{w}_*}(\un{w}, \un{w}_*)$ is of rank 1.
\end{lem}

\begin{proof} We know that $\Hom^{\not < \un{w}_*}(\un{w}, \un{w}_*)$
  has a basis given by light leaf morphisms corresponding to
  01-sequences $\un{e}'$ with $(\un{w}^{\un{e}'})_{\bullet} =
  \un{w}_*$. Notice that $(\un{w}^{\un{e}'})_{\bullet} =\un{w}_*$ implies that the number of $0$'s in $\un{e}'$ is less than or equal to $\ell(\un{w})-\ell(\un{w}_*)$. Hence $\defect(\un{e}') \ge \ell(\un{w}_*) -
    \ell(\un{w})$.  If the equality holds, then there are exactly $\ell(\un{w}_*)$ $U1$'s and exactly $\ell(\un{w})-\ell(\un{w}_*)$ $D0$'s in $\un{e}'$. In this case, $\un{e'}$ equals $\un{e}$ defined above. 
\end{proof}

\begin{lem} \label{lem:vanish}
  Suppose that $f : \un{w} \to \un{w}_*$ is a morphism in $\SD^{\not < \un{w}_*}$
  of degree $\ell(\un{w}_*) - \ell(\un{w})$. Let $\un{e}$ be a
  subsequence of $\un{w}$ such that $\un{e}$ has at least one $D0$ or
  $D1$. Fix a choice of light leaf morphism $LL_{\un{w},
    \un{e}} : \un{w} \to (\un{w}^{\un{e}})_\bullet$ in $\SD$. In $\SD^{\not < \un{w}_*}$ we have:
\begin{gather*}
  \begin{array}{c}
    \begin{tikzpicture}[scale=0.3]
      \draw (-2.5,4) -- (2.5,4) -- (4,0) -- (-4,0) -- (-2.5,4);
   \node at (0,2) {$f$};
       \foreach \x in {-2,-0.5} \draw[blue] (\x,4) -- (\x,4.5);
      \foreach \x in {-1.5,0.5,1.5} \draw[green] (\x,4) -- (\x,4.5);
      \foreach \x in {1,2} \draw[red] (\x,4) -- (\x,4.5);
      \foreach \x in {-1,0} \draw[purple] (\x,4) -- (\x,4.5);
      \draw (-3,-4) -- (3,-4) -- (4,0) -- (-4,0) -- (-3,-4);
   \node at (0,-2) {$\overline{LL_{\un{w}, \un{e}}}$};
       \foreach \x in {-2,-0.5} \draw[blue] (\x,-4) -- (\x,-4.5);
      \foreach \x in {-2.5,-1.5,0.5,1.5} \draw[green] (\x,-4) -- (\x,-4.5);
      \foreach \x in {1,2} \draw[red] (\x,-4) -- (\x,-4.5);
      \foreach \x in {-1,0,2.5} \draw[purple] (\x,-4) -- (\x,-4.5);
    \end{tikzpicture}
  \end{array} = 0.
\end{gather*}
\end{lem}

\begin{proof} When we precompose $f$ with a $2m$-valent vertex we
  obtain a morphism of the same form. If we precompose $f$ with a dot 
$\begin{array}{c}\tikz[xscale=0.3,yscale=-0.3]{\draw[dashed] (2,-1) rectangle (4,1);
\draw[color=blue] (3,-1) to (3,0);
\node[circle,fill,draw,inner sep=0mm,minimum size=1mm,color=blue] at
(3,0) {};}\end{array}$ to obtain a map $g : \un{w}' \to \un{w}
\stackrel{f}{\to} \un{w}_*$ then there are two cases:
\begin{enumerate}
\item if $\un{w}'_* = \un{w}_*$ then our new morphism $g$ is of the
  same form as $f$;
\item if $\un{w}'_* \ne \un{w}_*$ then $\un{w}'_* < \un{w}_*$ and
  our new morphism $g$ is zero in $\SD^{\not < \un{w}_*}$.
\end{enumerate}
Hence by our assumptions on $\un{e}$ we can assume that we precompose $f$ with a
trivalent vertex $\begin{array}{c}\tikz[xscale=0.3,yscale=0.3]{\draw[dashed] (-1,-1) rectangle (1,1);
\draw[color=blue] (-0.5,1) to (0,0) to (0.5,1);
\draw[color=blue] (0,0) to (0,-1);}\end{array}$ or cup $\begin{array}{c}\tikz[xscale=0.4,yscale=0.3]{\draw[dashed] (-1,-1) rectangle (1,1);
\draw[color=blue] (-0.5,1) to[out=-90,in=180] (0,-0.25)
to[out=0,in=-90] (0.5,1);}\end{array}$ (which factors through a
trivalent vertex).
Then the fact that these morphisms are 0
follows from the previous lemma by degree considerations.
\end{proof}

Let us say that an $S$-graph \emph{ends in a pitchfork} if it is
equivalent modulo isotopy and the Frobenius relations to an $S$-graph of the form
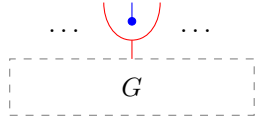
\[
  \begin{array}{c}
 \begin{tikzpicture}[scale=0.5]
      \foreach \x in {1.5,5} \node at (\x,-0.75) {$\dots$};
       \def\x{3.25}; \def\y{-1};
       \draw[red] (2.5,0) to[out=-90,in=180] (\x,\y) to[out=0,in=-90] (4,0);
       \draw[red] (\x,\y) -- (\x,-1.5);
       \draw[blue] (\x,0) -- (\x,-0.5);
 \node[circle,fill,draw,inner
      sep=0mm,minimum size=1mm,color=blue] at (\x,-0.5) {};
\draw[dashed,gray] (0,-1.5) rectangle (6.5,-3);
\node at (\x,-2.25) {$G$};
    \end{tikzpicture}
  \end{array}
\]
for some $S$-graph $G$. For example, the morphisms
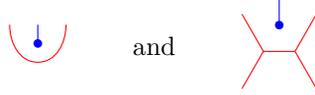
\begin{gather*}
   \begin{array}{c}
 \begin{tikzpicture}[scale=0.5]
       \def\x{3.25}; \def\y{-1};
       \draw[red] (2.5,0) to[out=-90,in=180] (\x,\y) to[out=0,in=-90] (4,0);
       \draw[blue] (\x,0) -- (\x,-0.5);
 \node[circle,fill,draw,inner
      sep=0mm,minimum size=1mm,color=blue] at (\x,-0.5) {};
    \end{tikzpicture}
  \end{array} \qquad \text{and} \qquad
   \begin{array}{c}
\tikz[scale=0.7]{
\draw[color=red] (-45:1cm) -- (0.3,0) -- (45:1cm);
\draw[color=red] (135:1cm) -- (-0.3,0) -- (-135:1cm);
\draw[color=red] (-0.3,0) -- (0.3,0);
 \draw[blue] (90:0.5) -- (90:1);
 \node[circle,fill,draw,inner
      sep=0mm,minimum size=1mm,color=blue] at (90:0.5)  {};
}
  \end{array} 
\end{gather*}
both end in pitchforks. The following lemma is proved in the same way
as the previous lemma.

\begin{lem} \label{lem:pitchfork}
  Suppose that $f : \un{w} \to \un{w}_*$ is a morphism in $\SD^{\not < \un{w}_*}$
  of degree $\ell(\un{w}_*) - \ell(\un{w})$. Then $f$ is ``killed by all
  pitchforks'': if $u_i u_{i+1} u_{i+2} = s t s $ for some $i$ and simple reflections
  $s, t \in S$ then:
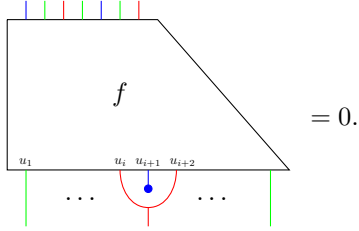
\begin{gather*}
  \begin{array}{c}
    \begin{tikzpicture}[scale=0.5]
      \draw (-0.5,4) -- (3.5,4) -- (7,0) -- (-0.5,0) -- (-0.5,4);
   \node at (2.5,2) {$f$};
       \foreach \x in {0,2} \draw[blue] (\x,4) -- (\x,4.5);
\foreach \x in {0.5,1.5,2.5} \draw[green] (\x,4) -- (\x,4.5);
      \foreach \x in {1,3} \draw[red] (\x,4) -- (\x,4.5);
       \foreach \x in {0,6.5} \draw[green] (\x,0) -- (\x,-1.5);
      \foreach \x in {1.5,5} \node at (\x,-0.75) {$\dots$};
       \def\x{3.25}; \def\y{-1};
       \draw[red] (2.5,0) to[out=-90,in=180] (\x,\y) to[out=0,in=-90] (4,0);
       \draw[red] (\x,\y) -- (\x,-1.5);
       \draw[blue] (\x,0) -- (\x,-0.5);
 \node[circle,fill,draw,inner
      sep=0mm,minimum size=1mm,color=blue] at (\x,-0.5) {};
\node[scale=0.5] at (0,0.2) {$u_1$};
\node[left,scale=0.5] at (2.8,0.2) {$u_i$};
\node[scale=0.5] at (\x,0.2) {$u_{i+1}$};
\node[right,scale=0.5] at (3.7,0.2) {$u_{i+2}$};
    \end{tikzpicture}
  \end{array} = 0.
\end{gather*}
\end{lem}




This lemma has as a consequence that that adding the ``square'' of a
$2m_{st}$-valent vertex does not change a gobbling morphism:

\begin{cor}
  Suppose that $f : \un{w} \to \un{w}_*$ is a morphism in $\SD^{\nless
   \un{w}_*}$ of degree $\ell(\un{w}_*) - \ell(\un{w})$. Let $i$ be such that $u_i u_{i+1} \dots u_{i+m_{st}-1} =
  sts \dots$. Then
\begin{gather*}
  \begin{array}{c}
    \begin{tikzpicture}[scale=0.5]
      \draw (-0.5,4) -- (3.5,4) -- (7,0) -- (-0.5,0) -- (-0.5,4);
   \node at (2.5,2) {$f$};
       \foreach \x in {0,2} \draw[blue] (\x,4) -- (\x,4.5);
\foreach \x in {0.5,1.5,2.5} \draw[brown] (\x,4) -- (\x,4.5);
      \foreach \x in {1,3} \draw[red] (\x,4) -- (\x,4.5);
       \foreach \x in {0,6.5} \draw[brown] (\x,0) -- (\x,-2);
      \foreach \x in {1.5,5} \node at (\x,-1) {$\dots$};
      \def\c{red};
      \foreach \x in {2.5,3.7}
      \draw[\c] (\x,-1) to[out=90,in=-90] (6.5-\x,0);
      \def\c{blue};
      \foreach \x in {2.8,4}
      \draw[\c] (\x,-1) to[out=90,in=-90] (6.5-\x,0);
      \def\c{blue};
      \foreach \x in {2.5,3.7}
      \draw[\c] (\x,-2) to[out=90,in=-90] (6.5-\x,-1);
      \def\c{red};
      \foreach \x in {2.8,4}
      \draw[\c] (\x,-2) to[out=90,in=-90] (6.5-\x,-1);
\foreach \y in {-0.15,-1,-1.85} \node[scale=0.5] at (3.25,\y) {$\dots$};
\node[scale=0.5] at (0,0.2) {$u_1$};
\node[left,scale=0.5] at (2.8,0.2) {$u_i$};
\node[right,scale=0.5] at (3.7,0.2) {$u_{i+m_{st}-1}$};
    \end{tikzpicture}
  \end{array} =
  \begin{array}{c}
    \begin{tikzpicture}[scale=0.5]
      \draw (-0.5,4) -- (3.5,4) -- (7,0) -- (-0.5,0) -- (-0.5,4);
   \node at (2.5,2) {$f$};
       \foreach \x in {0,2} \draw[blue] (\x,4) -- (\x,4.5);
\foreach \x in {0.5,1.5,2.5} \draw[brown] (\x,4) -- (\x,4.5);
      \foreach \x in {1,3} \draw[red] (\x,4) -- (\x,4.5);
       \foreach \x in {0,6.5} \draw[brown] (\x,0) -- (\x,-2);
       \foreach \x in {2.8,4} \draw[red] (\x,0) -- (\x,-2);
       \foreach \x in {2.5,3.7} \draw[blue] (\x,0) -- (\x,-2);
      \foreach \x in {1.5,5} \node at (\x,-1) {$\dots$};
\foreach \y in {-1} \node[scale=0.5] at (3.25,\y) {$\dots$};
\node[scale=0.5] at (0,0.2) {$u_1$};
\node[left,scale=0.5] at (2.8,0.2) {$u_i$};
\node[right,scale=0.5] at (3.7,0.2) {$u_{i+m_{st}-1}$};
    \end{tikzpicture}
  \end{array}.
\end{gather*}
\end{cor}

\begin{proof} As explained in \cite[Claim 6.7]{EliasCathedral}, two-color associativity \S\ref{2.7.6} and the Jones-Wenzl relation \S\ref{2.7.7} imply that we can write
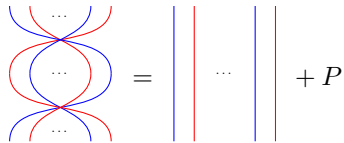
\begin{gather*}
  \begin{array}{c}
    \begin{tikzpicture}[scale=0.9]
      \def\c{red};
      \foreach \x in {2.5,3.7}
      \draw[\c] (\x,-1) to[out=90,in=-90] (6.5-\x,0);
      \def\c{blue};
      \foreach \x in {2.8,4}
      \draw[\c] (\x,-1) to[out=90,in=-90] (6.5-\x,0);
      \def\c{blue};
      \foreach \x in {2.5,3.7}
      \draw[\c] (\x,-2) to[out=90,in=-90] (6.5-\x,-1);
      \def\c{red};
      \foreach \x in {2.8,4}
      \draw[\c] (\x,-2) to[out=90,in=-90] (6.5-\x,-1);
\foreach \y in {-0.15,-1,-1.85} \node[scale=0.5] at (3.25,\y) {$\dots$};
    \end{tikzpicture}
  \end{array}
=
  \begin{array}{c}
 \begin{tikzpicture}[scale=0.9]
       \foreach \x in {2.8,4} \draw[red] (\x,0) -- (\x,-2);
       \foreach \x in {2.5,3.7} \draw[blue] (\x,0) -- (\x,-2);
       \node[scale=0.5] at (3.25,-1) {$\dots$};
    \end{tikzpicture}
  \end{array}+ P
\end{gather*}
where all terms in $P$ end in pitchforks. Now the result follows from Lemma \ref{lem:pitchfork}.
\end{proof}

\begin{prop} \label{cor:ll} If $\un{e}$ is defined as in
  \eqref{eq:edef}, then the images of any two choices of light
  leaves morphisms $L_{\un{w},\un{e}}$ and $L_{\un{w},\un{e}}'$ in
  $\SD^{\not < \un{w}_*}$ agree. 
\end{prop}

In the proof of the proposition we will need the notion of a rex move from \cite[\S 4.2]{EW2}: a rex
move $\b : \un{x} \to \un{x}'$ is a sequence of reduced expressions $\un{x} = \un{x}_0 \to
\un{x}_1 \to \dots \un{x}_m = \un{x}'$ such that each move $\un{x}_i
\to \un{x}_{i+1}$ involves the application of a single braid
relation ($st \dots \to ts \dots $ ($m_{st}$-factors) for
$m_{st} \ge 2$). Any two reduced expressions can be linked via rex
moves (a theorem of Matsumoto, see e.g. \cite[Theorem 1.9]{Lu}). Rex moves give morphisms in $\SD$ by composing the
corresponding $2m_{st}$-valent vertices (see \cite[\S 6.1]{EW2}).

\begin{proof}
  We prove the proposition by induction on the defect $\defect(\un{e})$ of
  $\un{e}$. If $\defect(\un{e}) = 0$ then a choice of light leaf
  morphism is simply a choice of reduced expression $\un{x}$ for
  $\un{w}_\bullet$ together with a rex move $\un{w} \to \un{x}$. The
  corollary follows in this case because the difference between any
  two rex moves lies in $\SD_{<
    \un{w}_*}$ (a consequence of Elias' Jones-Wenzl relation
  and the Zamolodchikov relations, see \cite[Lemma 7.4]{EW2}).

We now assume that $\defect(\un{e}) \ge 1$ and let $i+1$ denote the
position of the first $D0$ in $\un{e}$. We can draw two choices of light
leaves morphisms as
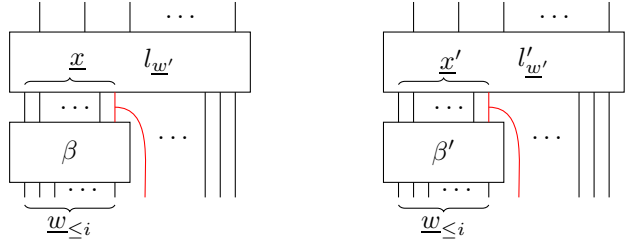
\[
\begin{array}{c} \begin{tikzpicture}[scale=0.4]
\draw (-0.5,0) rectangle (3.5,2); \node at (1.5,1) {$\b$}; 
\draw (-0.5,3) rectangle (7.5,5); \node at (4.5,4) {$l_{\un{w}'}$}; 
\foreach \x in {0,0.5,1,3} \draw(\x,-0.5) -- (\x,0); 
\node at (2,-0.25) {\dots};
\draw (7,-0.5) -- (7,3);\draw (6.5,-0.5) --(6.5,3);\draw (6,-0.5) --
(6,3);

\node at (5,1.5) {\dots};
\foreach \x in {0,0.5,2.5} \draw(\x,2) -- (\x,3); 
\draw[red] (3,2) -- (3,3);
\node at (1.75,2.5) {\dots};
\foreach \x in {0.5,2,3.5,7} \draw (\x,5) -- (\x,6); 

\draw[red] (3,2.5) to[out=0,in=90] (4,-0.5);

\node at (5.25,5.5) {\dots};
\draw [ decoration={brace,mirror,raise=0.1cm},decorate] (0,-0.5) to (3,-0.5);
\node at (1.5,-1.5) {$\un{w}_{\le i}$};
\draw [ decoration={brace,raise=0.1cm},decorate] (0,3) to (3,3);
\node at (1.75,4) {$\un{x}$};
\end{tikzpicture} \end{array}
\qquad \qquad
\begin{array}{c} \begin{tikzpicture}[scale=0.4]
\draw (-0.5,0) rectangle (3.5,2); \node at (1.5,1) {$\b'$}; 
\draw (-0.5,3) rectangle (7.5,5); \node at (4.5,4) {$l_{\un{w}'}'$}; 
\foreach \x in {0,0.5,1,3} \draw(\x,-0.5) -- (\x,0); 
\node at (2,-0.25) {\dots};
\draw (7,-0.5) -- (7,3);\draw (6.5,-0.5) --(6.5,3);\draw (6,-0.5) --
(6,3);

\node at (5,1.5) {\dots};
\foreach \x in {0,0.5,2.5} \draw(\x,2) -- (\x,3); 
\draw[red] (3,2) -- (3,3);
\node at (1.75,2.5) {\dots};
\foreach \x in {0.5,2,3.5,7} \draw (\x,5) -- (\x,6); 

\draw[red] (3,2.5) to[out=0,in=90] (4,-0.5);

\node at (5.25,5.5) {\dots};
\draw [ decoration={brace,mirror,raise=0.1cm},decorate] (0,-0.5) to (3,-0.5);
\node at (1.5,-1.5) {$\un{w}_{\le i}$};
\draw [ decoration={brace,raise=0.1cm},decorate] (0,3) to (3,3);
\node at (1.75,4) {$\un{x}'$};
\end{tikzpicture} \end{array}
\]
where $\un{x}, \un{x}'$ are reduced
expressions for $(\un{w}_{\le i})_\bullet$ both ending in
$s_{i+1}$ (coloured red in the above diagrams). By applying the previous corollary we can replace $\b'$ by
a rex move of the form:
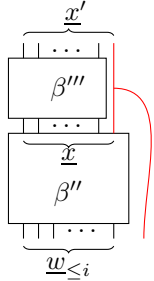
\[
\begin{array}{c} \begin{tikzpicture}[scale=0.4]
\draw (-0.5,-1) rectangle (3.5,2); \node at (1.5,0) {$\b''$}; 
\foreach \x in {0,0.5,1,3} \draw(\x,-1.5) -- (\x,-1); 
\node at (2,-1.25) {\dots};
\draw [ decoration={brace,mirror,raise=0.1cm},decorate] (0,-1.5) to (3,-1.5);
\node at (1.5,-2.5) {$\un{w}_{\le i}$};
\foreach \x in {0,0.5,2.5} \draw(\x,2) -- (\x,2.5); 
\node at (1.5,2.25) {\dots};
\draw (-0.5,2.5) rectangle (2.75,4.5); \node at (1.5,3.5) {$\b'''$}; 
\foreach \x in {0,0.5,2.5} \draw(\x,4.5) -- (\x,5); 
\node at (1.5,4.75) {\dots};
\draw[red] (3,2) to (3,5);
\draw[red] (3,3.5) to[out=0,in=90] (4,-1.5);
\draw [ decoration={brace,raise=0.1cm},decorate] (0,5) to (3,5);
\node at (1.75,6) {$\un{x}'$};
\draw [ decoration={brace,mirror,raise=0.1cm},decorate] (0,2) to
(3,2); 
\node at (1.5,1.2) {$\un{x}$};
\end{tikzpicture} \end{array}
\]
Now we can slide $\b'''$ into $l_{\un{w}'}'$ in the right hand diagram
above. Hence we can assume
$\un{x} = \un{x}'$ in the above diagram. Now
we can apply induction to conclude that $l_{\un{w}'} =
l_{\un{w}'}'$. Thus we can assume that the above two diagrams are
identical except in the boxes corresponding to $\b$ and $\b'$.
Now the result follows from Lemma \ref{lem:vanish} and the lemma below.
\end{proof}

\begin{lem}
  Let $\b_1, \b_2$ be two rex moves from $\un{w}_1 \to \un{w}_2$,
  where $\un{w}_1$ and $\un{w}_2$ are reduced expressions. Then we can
  write
\begin{gather*}
  \begin{array}{c}
\begin{tikzpicture}[scale=0.4]
   \def\y{1.5}; \def\yt{2};
       \foreach \x in {-2,-0.5} \draw[blue] (\x,\y) -- (\x,\yt);
      \foreach \x in {-1.5,0.5,1.5} \draw[green] (\x,\y) -- (\x,\yt);
      \foreach \x in {1,2} \draw[red] (\x,\y) -- (\x,\yt);
      \foreach \x in {-1,0} \draw[purple] (\x,\y) -- (\x,\yt);
   \def\y{-1.5}; \def\yt{-2};
       \foreach \x in {2,0.5} \draw[blue] (\x,\y) -- (\x,\yt);
      \foreach \x in {1.5,-0.5,-1.5} \draw[green] (\x,\y) -- (\x,\yt);
      \foreach \x in {-1,-2} \draw[red] (\x,\y) -- (\x,\yt);
      \foreach \x in {1,0} \draw[purple] (\x,\y) -- (\x,\yt);
\draw (-2.5,\y) rectangle (2.5,-1*\y); \node at (0,0) {$\b_1$}; 
  \end{tikzpicture}
\end{array}-
  \begin{array}{c}
\begin{tikzpicture}[scale=0.4]
   \def\y{1.5}; \def\yt{2};
       \foreach \x in {-2,-0.5} \draw[blue] (\x,\y) -- (\x,\yt);
      \foreach \x in {-1.5,0.5,1.5} \draw[green] (\x,\y) -- (\x,\yt);
      \foreach \x in {1,2} \draw[red] (\x,\y) -- (\x,\yt);
      \foreach \x in {-1,0} \draw[purple] (\x,\y) -- (\x,\yt);
   \def\y{-1.5}; \def\yt{-2};
       \foreach \x in {2,0.5} \draw[blue] (\x,\y) -- (\x,\yt);
      \foreach \x in {1.5,-0.5,-1.5} \draw[green] (\x,\y) -- (\x,\yt);
      \foreach \x in {-1,-2} \draw[red] (\x,\y) -- (\x,\yt);
      \foreach \x in {1,0} \draw[purple] (\x,\y) -- (\x,\yt);
\draw (-2.5,\y) rectangle (2.5,-1*\y); \node at (0,0) {$\b_2$}; 
  \end{tikzpicture}
\end{array}
=
\sum_{\un{e}, \un{e}'} c_{\un{e}, \un{e}'}
  \begin{array}{c}
\begin{tikzpicture}[scale=0.4]
   \def\y{2}; \def\yt{2.5};
       \foreach \x in {-2,-0.5} \draw[blue] (\x,\y) -- (\x,\yt);
      \foreach \x in {-1.5,0.5,1.5} \draw[green] (\x,\y) -- (\x,\yt);
      \foreach \x in {1,2} \draw[red] (\x,\y) -- (\x,\yt);
      \foreach \x in {-1,0} \draw[purple] (\x,\y) -- (\x,\yt);
   \def\y{-2}; \def\yt{-2.5};
       \foreach \x in {2,0.5} \draw[blue] (\x,\y) -- (\x,\yt);
      \foreach \x in {1.5,-0.5,-1.5} \draw[green] (\x,\y) -- (\x,\yt);
      \foreach \x in {-1,-2} \draw[red] (\x,\y) -- (\x,\yt);
      \foreach \x in {1,0} \draw[purple] (\x,\y) -- (\x,\yt);
\draw (-2.5,\y) -- (-1.5,0) -- (1.5,0) -- (2.5,\y) -- (-2.5,\y);
\node at (0,-1) {$LL_{\un{w}_1, \un{e}}$} ;
\draw (-2.5,-1*\y) -- (-1.5,0) -- (1.5,0) -- (2.5,-1*\y) --
(-2.5,-1*\y);
\node at (0,1) {$\overline{LL_{\un{w}_2, \un{e}'}}$};
  \end{tikzpicture}
\end{array}
\end{gather*}
for homogenous $c_{\un{e}, \un{e}'} \in R$ and where $LL_{\un{w}_1,
  \un{e}}$, $LL_{\un{w}_2, \un{e}'}$ are light
leaf morphisms corresponding to subsequences $\un{e}, \un{e}'$
having at least one $D1$ or $D0$.
\end{lem}

\begin{proof}
  We know that the left hand side vanishes in $\SD^{\not <
    {\un{w}_1}_\bullet}$ and hence all terms on the right hand side factor
  through some $\un{y}$ with $\un{y}_\bullet <
  (\un{w}_1)_\bullet$. Assume for contradiction that there is a term
  on the right hand side with  $c_{\un{e}, \un{e}'} \ne 0$, where
  $\un{e}$ is a subsequence of   $\un{w}_1$ without $D1$ or $D0$. In
  particular, $\defect(\un{e}) = \ell(\un{w}_1) - \ell(
  (\un{w_1}^{\un{e}})_\bullet)$. For degree reasons $\defect(\un{e}') +
  \deg(c_{\un{e}, \un{e}'}) = \ell(
  (\un{w_1}^{\un{e}})_\bullet) - \ell(\un{w}_1)$. However  $\defect(\un{e}') \ge \ell(
  (\un{w_1}^{\un{e}})_\bullet) - \ell(\un{w}_1)$ and hence
  $\deg(c_{\un{e}, \un{e}'}) = 0$ and $\defect(\un{e}') = \ell(
  (\un{w_1}^{\un{e}})_\bullet) - \ell(\un{w}_1)$. Hence $\un{e'}$
  consists entirely of $U1$'s and $D0$'s, however this is impossible
  as $\un{w}_2$ is reduced.
\end{proof}

\subsection{Proof of Proposition \ref{gobblenew} in rank 2.}

Assume that $W$ is of rank 2, i.e. a dihedral
group. We denote the simple reflections of $W$ by $s$ and $t$ and
denote the order of $st$ (possibly $\infty$) by $m$. The
longest element of $W$ is $w_0$ (if it exists).

If $m = \infty$ then the proposition is easy: only trivalent
vertices can occur, and the proposition follows from the fact that any
two morphisms
\[
ss \dots s \to s \quad \text{($k$ factors on the left)}
\]
consisting of exactly $k$ trivalent vertices and no dots are
equal (a consequence of the Frobenius associativity relation \S\ref{2.7.2}).

From now on we will assume that $m$ is finite. If
$\un{w}_* \ne w_0$ then the result follows similarly to the case $m =
\infty$ above. So we may assume $\un{w}_* = w_0$.

Let us denote by $l_{\un{w}}$ the light leaves map $\un{w} \to
\un{w}_*$ as above. In this section (i.e. until \S \ref{proof:gen}) we
always regard $l_{\un{w}}$ as a morphism to $b_{w_0}$. That is we
always compose with some number $\ge 1$ of iterates of the
$2_{m_{st}}$-valent vertex.

We can depict $l_{\un{w}}$ schematically as follows (we depict
the case $m = 4$):
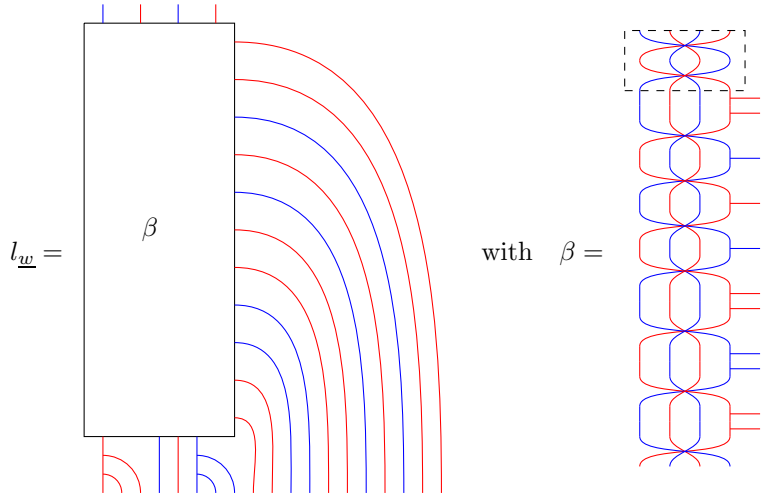
\[
l_{\un{w}} = 
\begin{array}{c} \begin{tikzpicture}[scale=0.5]
\draw (1.5,10.5) -- (5.5,10.5) -- (5.5,-0.5) -- (1.5,-0.5) --
(1.5,10.5);
\node at (3.25,5) {$\b$};

\coordinate (a0) at (5.5,0) {};
\coordinate (a1) at (5.5,1) {};
\coordinate (a2) at (5.5,2) {};
\coordinate (a3) at (5.5,3) {};
\coordinate (a4) at (5.5,4) {};
\coordinate (a5) at (5.5,5) {};
\coordinate (a6) at (5.5,6) {};
\coordinate (a7) at (5.5,7) {};
\coordinate (a8) at (5.5,8) {};
\coordinate (a9) at (5.5,9) {};
\coordinate (a10) at (5.5,10) {};

\coordinate (b0) at (6,-2) {};
\coordinate (b1) at (6.5,-2) {};
\coordinate (b2) at (7,-2) {};
\coordinate (b3) at (7.5,-2) {};
\coordinate (b4) at (8,-2) {};
\coordinate (b5) at (8.5,-2) {};
\coordinate (b6) at (9,-2) {};
\coordinate (b7) at (9.5,-2) {};
\coordinate (b8) at (10,-2) {};
\coordinate (b9) at (10.5,-2) {};
\coordinate (b10) at (11,-2) {};

\draw[red] (a0) to [out=0,in=90] (b0);
\draw[red] (a1) to [out=0,in=90] (b1);
\draw[blue] (a2) to [out=0,in=90] (b2);
\draw[blue] (a3) to [out=0,in=90] (b3);
\draw[red] (a4) to [out=0,in=90] (b4);
\draw[red] (a5) to [out=0,in=90] (b5);
\draw[blue] (a6) to [out=0,in=90] (b6);
\draw[red] (a7) to [out=0,in=90] (b7);
\draw[blue] (a8) to [out=0,in=90] (b8);
\draw[red] (a9) to [out=0,in=90] (b9);
\draw[red] (a10) to [out=0,in=90] (b10);

\draw[red] (2, -0.5) -- (2,-2);
\draw[blue] (3.5, -0.5) -- (3.5,-2);
\draw[red] (4, -0.5) -- (4,-2);
\draw[blue] (4.5, -0.5) -- (4.5,-2);

\draw[red] (2,-1) to [out=0, in=90] (3,-2);
\draw[red] (2,-1.5) to [out=0, in=90] (2.5,-2);

\draw[blue] (4.5,-1) to [out=0, in=90] (5.5,-2);
\draw[blue] (4.5,-1.5) to [out=0, in=90] (5,-2);


\draw[blue] (2,10.5) -- (2,11);
\draw[red] (3,10.5) -- (3,11);
\draw[blue] (4,10.5) -- (4,11);
\draw[red] (5,10.5) -- (5,11);
\end{tikzpicture}\end{array}
\quad \text{with} \quad
\b = 
\begin{array}{c} \begin{tikzpicture}[xscale=0.4,yscale=0.2]
\def\A{10}

\foreach \y in {0,8,15,21,27}
{
\draw[red] (0,\y) to [out=90,in=270] (3,\y+2);
\draw[blue] (1,\y) to [out=90,in=270] (2,\y+2);
\draw[red] (2,\y) to [out=90,in=270] (1,\y+2);
\draw[blue] (3,\y) to [out=90,in=270] (0,\y+2);
}

\foreach \y in {4,12,18,25}
{
\draw[blue] (0,\y) to [out=90,in=270] (3,\y+2);
\draw[red] (1,\y) to [out=90,in=270] (2,\y+2);
\draw[blue] (2,\y) to [out=90,in=270] (1,\y+2);
\draw[red] (3,\y) to [out=90,in=270] (0,\y+2);
}

\foreach \y in {2,3,10,11,17,23,24}
{
\draw[blue] (0,\y) to [out=90,in=270] (0,\y+1);
\draw[red] (1,\y) to [out=90,in=270] (1,\y+1);
\draw[blue] (2,\y) to [out=90,in=270] (2,\y+1);
\draw[red] (3,\y) to [out=90,in=270] (3,\y+1);
\draw[red] (3,\y+0.5) to (4,\y+0.5);
}

\foreach \y in {6,7,14,20}
{
\draw[red] (0,\y) to [out=90,in=270] (0,\y+1);
\draw[blue] (1,\y) to [out=90,in=270] (1,\y+1);
\draw[red] (2,\y) to [out=90,in=270] (2,\y+1);
\draw[blue] (3,\y) to [out=90,in=270] (3,\y+1);
\draw[blue] (3,\y+0.5) to (4,\y+0.5);
}

\draw[dashed] (-0.5,25) rectangle (3.5,29);
\end{tikzpicture} \end{array}
\]
(The dashed region in $\b$ denotes any $k \ge 1$ compositions of
$2m$-valent vertices which will be implicit (and not displayed) in all
diagrams until \S \ref{proof:gen}.)

\begin{lem}  \label{lem:ga}
Let $\un{w} \in Ex(S)$ with $\un{w}_* = w_0$. In $\Hom(\un{w},
b_{w_0})$ we have:
  \begin{enumerate}
  \item \begin{gather*}
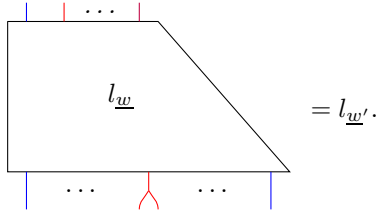
  \begin{array}{c}
    \begin{tikzpicture}[scale=0.5]
      \draw (-0.5,4) -- (3.5,4) -- (7,0) -- (-0.5,0) -- (-0.5,4);
   \node at (2.5,2) {$l_{\un{w}}$};
       \foreach \x in {0} \draw[blue] (\x,4) -- (\x,4.5);
      \foreach \x in {1} \draw[red] (\x,4) -- (\x,4.5);
      \draw[purple] (3,4) -- (3,4.5);
      \node at (2,4.25) {\dots};
       \foreach \x in {0,6.5} \draw[blue] (\x,0) -- (\x,-1);
      \foreach \x in {1.5,5} \node at (\x,-0.5) {$\dots$};
     \coordinate (mid) at (3.25,-0.5);
     \draw[red] (3,-1) to [out=90, in = -120] (mid);
     \draw[red] (3.5,-1) to [out=90, in = -60] (mid);
     \draw[red] (3.25,0) to (mid);
    \end{tikzpicture}
  \end{array} = l_{\un{w}'}.
\end{gather*}
  \item 
  If $\un{w}$ is of the form $\un{w} = u_1 \dots u_n$ with $u_i
  u_{i+1} \dots u_{i+m-1} = s t s \dots $ ($m$ times). Then,  in $\Hom(\un{w},
b_{w_0})$:
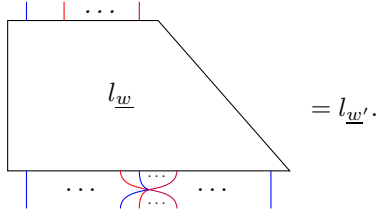
\begin{gather*}
  \begin{array}{c}
    \begin{tikzpicture}[scale=0.5]
      \draw (-0.5,4) -- (3.5,4) -- (7,0) -- (-0.5,0) -- (-0.5,4);
   \node at (2.5,2) {$l_{\un{w}}$};
       \foreach \x in {0} \draw[blue] (\x,4) -- (\x,4.5);
      \foreach \x in {1} \draw[red] (\x,4) -- (\x,4.5);
      \draw[purple] (3,4) -- (3,4.5);
      \node at (2,4.25) {\dots};
       \foreach \x in {0,6.5} \draw[blue] (\x,0) -- (\x,-1);
      \foreach \x in {1.5,5} \node at (\x,-0.5) {$\dots$};
      \coordinate (c) at (3.25,-0.5);
      \def\c{red}; 
      \draw[\c] (2.5,0) to[out=-90,in=170] (c);
      \draw[\c] (3,-1) to[out=90,in=-160] (c);
      \def\c{blue};
      \draw[\c] (2.5,-1) to[out=90,in=-170] (c);
      \draw[\c] (3,0) to[out=-90,in=160] (c);
      \def\c{purple};
      \draw[\c] (4,-1) to[out=90,in=-10] (c);
      \draw[\c] (4,0) to[out=-90,in=10] (c);
      \foreach \y in{-0.15,-0.85} \node[scale=0.6] at (3.5,\y) {$\dots$};
      \def\c{blue};
    \end{tikzpicture}
  \end{array} = l_{\un{w}'}.
\end{gather*}
  \end{enumerate}
(In (1) and (2) the definition of $\un{w}'$ should be clear from the diagrams.)
\end{lem}

This lemma is clearly equivalent to Proposition \ref{gobblenew}. (The
apparently missing case of composing with a splitting trivalent vertex $\begin{array}{c}\tikz[xscale=0.3,yscale=0.3]{\draw[dashed] (-1,-1) rectangle (1,1);
\draw[color=red] (-0.5,1) to (0,0) to (0.5,1);
\draw[color=red] (0,0) to (0,-1);}\end{array}$ 
does not satisfy the conditions of Proposition \ref{gobble}.)
Also,
part (1) of the lemma follows directly from the definitions and the
Frobenius associativity relation \S\ref{2.7.2}. After examining the above picture of
$l_{\un{w}}$ it is not difficult to see that (2) is implied by the
following two claims:

\begin{claim} We have:
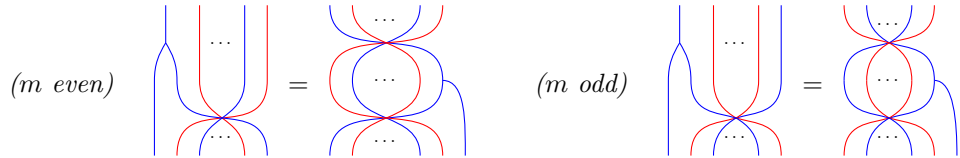
\[
\text{($m$ even)} \quad
\begin{array}{c}\begin{tikzpicture}[xscale=0.3, yscale=0.5]
\foreach \x in {1,4} \draw[red] (\x,0) to[out=90,in=-90] (6-\x,2) to (6-\x,4);
\foreach \x in {2} \draw[blue] (\x,0) to[out=90,in=-90] (6-\x,2) to (6-\x,4);
\draw[blue] (5,0) to[out=90,in=-90] (1,2) to[out=90,in=-45] (0.5,3) to[out=90,in=-90] (0.5,4);
\draw[blue] (0,0) to (0,2) to[out=90,in=-135] (0.5,3);
\foreach\y in {0.5,3} \node[scale=0.7] at (3,\y) {$\dots$};
\end{tikzpicture}\end{array}
=
\begin{array}{c}\begin{tikzpicture}[xscale=0.3, yscale=0.4]
 \foreach\x in{0,4} \draw[blue] (\x,0) to[out=90,in=-90] (5-\x,2.5) to (5-\x,2.5) to[out=90,in=-90] (\x,5);
 \foreach\x in{1,5} \draw[red] (\x,0) to[out=90,in=-90] (5-\x,2.5) to (5-\x,2.5) to[out=90,in=-90] (\x,5);
\draw[blue] (5,2.5) to[out=0,in=90] (6,0);
\foreach\y in {0.5,2.5,4.5} \node[scale=0.7] at (2.5,\y) {$\dots$};
\end{tikzpicture}\end{array}
\qquad
\text{($m$ odd)} \quad
\begin{array}{c}\begin{tikzpicture}[xscale=0.3, yscale=0.5]
\draw[red] (1,0) to[out=90,in=-170] (3,1);\draw[red] (5,0) to[out=90,in=-10] (3,1);
\draw[blue] (5,4) to (5,2) to[out=-90,in=10] (3,1);
\draw[red] (2,4) to (2,2) to[out=-90,in=160] (3,1);\draw[red] (4,4) to (4,2) to[out=-90,in=20] (3,1);
\draw[blue] (2,0) to[out=90,in=-160] (3,1);\draw[blue] (4,0) to[out=90,in=-20] (3,1);
\draw[blue] (3,1) to[out=170,in=-90] (1,2) to[out=90,in=-45] (0.5,3) to[out=90,in=-90] (0.5,4);
\draw[blue] (0,0) to (0,2) to[out=90,in=-135] (0.5,3);
\foreach\y in {0.5,3} \node[scale=0.7] at (3,\y) {$\dots$};
\end{tikzpicture}\end{array}
=
\begin{array}{c}\begin{tikzpicture}[xscale=0.3, yscale=0.5]
\def\y{0}; 
\draw[red] (1,\y) to[out=90,in=-170] (3,1);\draw[red] (5,\y) to[out=90,in=-10] (3,1);
\draw[blue] (5,\y+2) to[out=-90,in=10] (3,1);
\draw[red] (2,\y+2) to[out=-90,in=160] (3,1);\draw[red] (4,\y+2) to[out=-90,in=20] (3,1);
\draw[blue] (2,\y) to[out=90,in=-160] (3,1);\draw[blue] (4,\y) to[out=90,in=-20] (3,1);
\draw[blue] (3,1) to[out=170,in=-90] (1,\y+2);
\def\y{2}; 
\draw[blue] (1,\y) to[out=90,in=-170] (3,\y+1);\draw[blue] (5,\y) to[out=90,in=-10] (3,\y+1);
\draw[red] (5,\y+2) to[out=-90,in=10] (3,\y+1);
\draw[blue] (2,\y+2) to[out=-90,in=160] (3,\y+1);\draw[blue] (4,\y+2) to[out=-90,in=20] (3,\y+1);
\draw[red] (2,\y) to[out=90,in=-160] (3,\y+1);\draw[red] (4,\y) to[out=90,in=-20] (3,\y+1);
\draw[red] (3,\y+1) to[out=170,in=-90] (1,\y+2);
\draw[blue] (5,2) to[out=0,in=90] (6,0);
\foreach\y in {0.5,2,3.5} \node[scale=0.7] at (3,\y) {$\dots$};
\end{tikzpicture}\end{array}
\]
\end{claim}

\begin{claim} \label{claim2m}
Suppose $k \le m$ and $\un{w}$ is of the form
\[
\un{w} = sts\dots \quad \text{($k+m$-factors)}
\]
Then in $\Hom(\un{w}',
b_{w_0})$ we have:
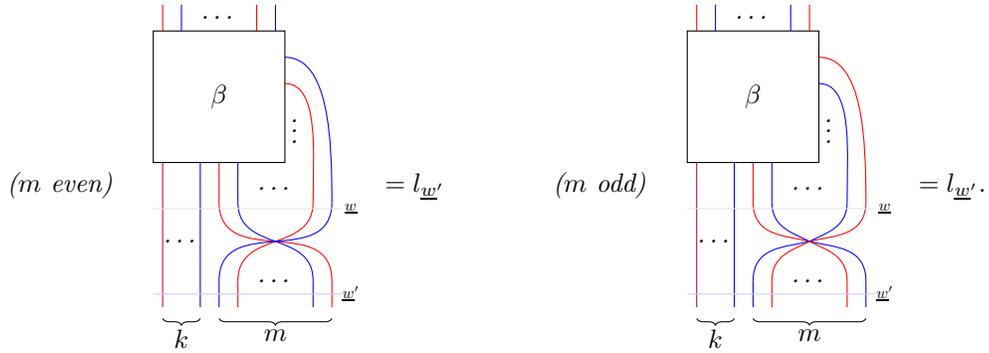
\begin{gather*}
\text{($m$ even)} \quad
  \begin{array}{c}
    \begin{tikzpicture}[xscale=0.25,yscale=0.35]
      \def\y{5};
     \def\z{-1.5};
\def\c{-3}
 \def\w{-4.5};
\draw (-0.5,0) rectangle (6.5,\y);
\draw[purple] (0,0) -- (0,\w-1);
 \foreach\x in {0,5} \draw[red] (\x,\y) -- (\x,\y+1); 
 \foreach\x in {1,6} \draw[blue] (\x,\y) -- (\x,\y+1); 
\node at (3,\y+0.5) {\dots};
\draw[blue] (6.5,\y-1) to[out=0,in=90] (9,\z); 
\draw[red] (6.5,\y-2) to[out=0,in=90] (8,\z);
\draw[blue] (8,\w) -- (8,\w-1);\draw[red] (9,\w) -- (9,\w-1);
\node at (7,1.5) {\vdots};
\draw[blue] (2,0) to (2,\w-1);
\draw[red] (3,0) to (3,\z); \draw[blue] (3,\w) to (3,\w-1);
\draw[blue] (4,0) to (4,\z);\draw[red] (4,\w) to (4,\w-1);
\node at (1,\c) {\dots};
\node at (6,-1) {\dots}; \node at (6,-4.5) {\dots};
\foreach \x in {3,8} \draw[red] (\x,\z) to[out=-90,in=90] (12-\x,\w);
 \foreach \x in {4,9} \draw[blue] (\x,\z) to[out=-90,in=90] (12-\x,\w);
\draw [ decoration={brace,mirror,raise=0.5cm},decorate] (0,\w) --(2,\w);
\draw [ decoration={brace,mirror,raise=0.5cm},decorate] (3,\w) --(9,\w);
\node at (3,\y/2) {$\b$}; \node[below] at (1,\w-1.5) {$k$}; \node[below] at (6,\w-1.5) {$m$};
\draw[gray!20!white] (-0.5,\z/2-1) -- (9.5,\z/2-1); \node[scale=0.6] at (10,\z/2-1) {$\un{w}$};
\draw[blue!20!white] (-0.5,\w-0.5) -- (9.5,\w-0.5);\node[scale=0.6] at (10,\w-0.5) {$\un{w}'$};
    \end{tikzpicture}
  \end{array}= l_{\un{w}'}
\qquad \qquad
\text{($m$ odd)} \quad
  \begin{array}{c}
    \begin{tikzpicture}[xscale=0.25,yscale=0.35]
      \def\y{5};
     \def\z{-1.5};
\def\c{-3}
 \def\w{-4.5};
\draw (-0.5,0) rectangle (6.5,\y);
\draw[purple] (0,0) -- (0,\w-1);
 \foreach\x in {0,6} \draw[red] (\x,\y) -- (\x,\y+1); 
 \foreach\x in {1,5} \draw[blue] (\x,\y) -- (\x,\y+1); 
\node at (3,\y+0.5) {\dots};
\draw[red] (6.5,\y-1) to[out=0,in=90] (9,\z); 
\draw[blue] (6.5,\y-2) to[out=0,in=90] (8,\z);
\draw[red] (8,\w) -- (8,\w-1);\draw[blue] (9,\w) -- (9,\w-1);
\node at (7,1.5) {\vdots};
\draw[blue] (2,0) to (2,\w-1);
\draw[red] (3,0) to (3,\z); \draw[blue] (3,\w) to (3,\w-1);
\draw[blue] (4,0) to (4,\z);\draw[red] (4,\w) to (4,\w-1);
\node at (1,\c) {\dots};
\node at (6,-1) {\dots}; \node at (6,-4.5) {\dots};
\draw[red] (3,\z) to[out=-90,in=170] (6,\c);\draw[red] (9,\z) to[out=-90,in=10] (6,\c);
\draw[blue] (3,\w) to[out=90,in=-170] (6,\c); \draw[blue] (9,\w) to[out=90,in=-10] (6,\c);
\draw[blue] (4,\z) to[out=-90,in=160] (6,\c);\draw[blue] (8,\z) to[out=-90,in=20] (6,\c);
\draw[red] (4,\w) to[out=90,in=-160] (6,\c); \draw[red] (8,\w) to[out=90,in=-20] (6,\c);
\draw [ decoration={brace,mirror,raise=0.5cm},decorate] (0,\w) --(2,\w);
\draw [ decoration={brace,mirror,raise=0.5cm},decorate] (3,\w) --(9,\w);
\node at (3,\y/2) {$\b$}; \node[below] at (1,\w-1.5) {$k$}; \node[below] at (6,\w-1.5) {$m$};
\draw[gray!20!white] (-0.5,\z/2-1) -- (9.5,\z/2-1); \node[scale=0.6] at (10,\z/2-1) {$\un{w}$};
\draw[blue!20!white] (-0.5,\w-0.5) -- (9.5,\w-0.5);\node[scale=0.6] at (10,\w-0.5) {$\un{w}'$};
    \end{tikzpicture}
  \end{array} = l_{\un{w}'}.
\end{gather*}
(In both cases the purple line indicates a line that is either blue or red, depending on the parity of $k$.)
\end{claim}

The first claim is just a restatement of the two-colour
associativity relation \S\ref{2.7.6}. It remains to prove Claim \ref{claim2m}.

We prove Claim \ref{claim2m} by induction on $k$. For $k = 0$ there is
nothing to prove. The case $k = 1$ is again a restatement of two-colour
associativity relation \S\ref{2.7.6}. For concreteness we assume that $m$ is even. It is not difficult to check that
the same argument works for $m$ odd.

Assume first that $k$ is even. Then we can write
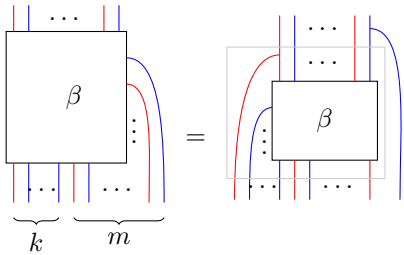
\[
  \begin{array}{c}
    \begin{tikzpicture}[xscale=0.2,yscale=0.35]
      \def\y{5};
     \def\z{-1.5};
\def\c{-3}
 \def\w{-4.5};
\draw (-1.5,0) rectangle (6.5,\y); 
 \foreach\x in {-1,5} \draw[red] (\x,\y) -- (\x,\y+1); 
 \foreach\x in {0,6} \draw[blue] (\x,\y) -- (\x,\y+1); 
\node at (2.5,\y+0.5) {\dots}; 
\draw[blue] (6.5,\y-1) to[out=0,in=90] (9,\z); 
\draw[red] (6.5,\y-2) to[out=0,in=90] (8,\z);
\foreach\x in {0,2,4} \draw[blue] (\x,0) to (\x,\z); 
\foreach\x in {-1,3} \draw[red] (\x,0) to (\x,\z); 
\node at (1,-1) {\dots}; 
\node at (6,-1) {\dots}; 
\node at (7,\y/2-1) {\vdots}; 
\draw [ decoration={brace,mirror,raise=0.2cm},decorate] (-1,\z) --(2,\z);
\draw [ decoration={brace,mirror,raise=0.2cm},decorate] (3,\z) --(9,\z);
\node at (3,\y/2) {$\b$}; \node[below] at (0.5,\z-0.8) {$k$}; \node[below] at (6,\z-0.8) {$m$};
    \end{tikzpicture}
  \end{array}
= 
  \begin{array}{c}\begin{tikzpicture}[xscale=0.2,yscale=0.35]
      \def\y{3};
     \def\z{-1.5};
\def\c{-3}
 \def\w{-4.5};
\draw (-0.5,0) rectangle (6.5,\y); \node at (3,\y/2) {$\b$};
 \foreach\x in {0,5} \draw[red] (\x,\y) -- (\x,\y+2.5); 
 \foreach\x in {1,6} \draw[blue] (\x,\y) -- (\x,\y+2.5); 
\node at (3,\y+2) {\dots}; \node at (3,\y+0.7) {\dots}; 
\draw[blue] (6,\y+2) to[out=0,in=90] (8,\z); 
\draw[blue] (-0.5,\y-1) to[out=180,in=90] (-2,\z); 
\draw[red] (0,\y+1) to[out=180,in=90] (-3,\z); 
\foreach\x in {0,2} \draw[blue] (\x,0) to (\x,\z); 
\foreach\x in {1,6} \draw[red] (\x,0) to (\x,\z); 
\node at (-1,-1) {\dots};  \node at (4,-1) {\dots};  
\node at (-1,\y/2-0.5) {\vdots}; 
\node[below,white] at (3,\z-0.8) {$m$}; 
\draw[gray!40!white] (7,\z+0.8) rectangle (-3.5,\y+1.3) ;
    \end{tikzpicture} \end{array}
\]
(we use induction and Proposition \ref{gobblenew} to rewrite
the term on the right hand side enclosed in the gray box as a light
leaves morphism). Hence the term in the claim is equal to:
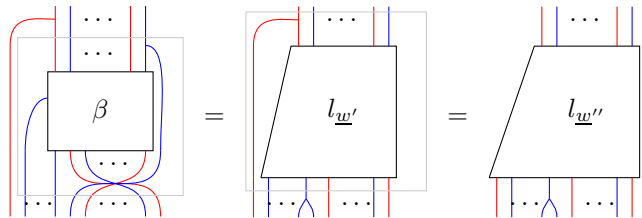
\begin{gather*}
  \begin{array}{c}\begin{tikzpicture}[xscale=0.2,yscale=0.35]
      \def\y{3};
     \def\z{0};
\def\c{-3}
 \def\w{-2.5};
\draw (-0.5,0) rectangle (6.5,\y); \node at (3,\y/2) {$\b$};
 \foreach\x in {0,5} \draw[red] (\x,\y) -- (\x,\y+2.5); 
 \foreach\x in {1,6} \draw[blue] (\x,\y) -- (\x,\y+2.5); 
\node at (3,\y+2) {\dots}; \node at (3,\y+0.7) {\dots}; 
\draw[blue] (6,\y+1) to[out=0,in=90] (7,\z); 
\draw[blue] (-0.5,\y-1) to[out=180,in=90] (-2,\z); 
\draw[red] (0,\y+2) to[out=180,in=90] (-3,\y+1) to[out=-90,in=90] (-3,\z); 
\foreach\x in {0,2} \draw[blue] (\x,0) to (\x,\z); 
\foreach\x in {1,6} \draw[red] (\x,0) to (\x,\z); 
\node at (-1,-2) {\dots};  \node at (4,-2) {\dots}; \node at (4,-0.5) {\dots};  
\foreach \x in {1,6} \draw[red] (\x,\z) to[out=-90,in=90] (8-\x,\w); 
\foreach \x in {2,7} \draw[blue] (\x,\z) to[out=-90,in=90] (8-\x,\w);
\foreach \x in {0,-2} \draw[blue] (\x,\z) to (\x,\w);\foreach \x in {-3} \draw[red] (\x,\z) to (\x,\w);
\draw[gray!40!white] (8.5,\w+0.8) rectangle (-2.5,\y+1.3) ;
    \end{tikzpicture} \end{array}
= 
 \begin{array}{c}\begin{tikzpicture}[xscale=0.2,yscale=0.35]
      \def\y{4};
     \def\z{0};
\def\d{-1}
\def\c{-3}
 \def\w{-2.5};
\draw (-2.5,\d) -- (6.5,\d) -- (6.5,\y) -- (-0.5,\y) --(-2.5,\d);\node at (3,\y/2-0.5) {$l_{\un{w}'}$};
 \foreach\x in {0,5} \draw[red] (\x,\y) -- (\x,\y+1.5); 
 \foreach\x in {1,6} \draw[blue] (\x,\y) -- (\x,\y+1.5); 
\node at (3,\y+1) {\dots}; 
\draw[red] (0,\y+1) to[out=180,in=90] (-3,\y) to[out=-90,in=90] (-3,\d); 
\coordinate (c) at (0.5,-1.8);
\draw[blue] (0,\w) to[out=90,in=-135] (c); \draw[blue] (1,\w) to[out=90,in=-45] (c); \draw[blue] (c) to (0.5,\d);
\foreach\x in {-3,2,6} \draw[red] (\x,\w) to (\x,\d); 
\foreach\x in {-2,5} \draw[blue] (\x,\w) to (\x,\d); 
\node at (-1,-2) {\dots};  \node at (4,-2) {\dots}; 
\draw[gray!40!white] (8.5,\d-0.5) rectangle (-3.5,\y+1.3) ;
    \end{tikzpicture} \end{array}
=
 \begin{array}{c}\begin{tikzpicture}[xscale=0.2,yscale=0.35]
      \def\y{4};
     \def\z{0};
\def\d{-1}
\def\c{-3}
 \def\w{-2.5};
\draw (-3.5,\d) -- (6.5,\d) -- (6.5,\y) -- (-0.5,\y) --(-3.5,\d);\node at (3,\y/2-0.5) {$l_{\un{w}''}$};
 \foreach\x in {0,5} \draw[red] (\x,\y) -- (\x,\y+1.5); 
 \foreach\x in {1,6} \draw[blue] (\x,\y) -- (\x,\y+1.5); 
\node at (3,\y+1) {\dots}; 
\coordinate (c) at (0.5,-1.8);
\draw[blue] (0,\w) to[out=90,in=-135] (c); \draw[blue] (1,\w) to[out=90,in=-45] (c); \draw[blue] (c) to (0.5,\d);
\foreach\x in {-3,2,6} \draw[red] (\x,\w) to (\x,\d); 
\foreach\x in {-2,5} \draw[blue] (\x,\w) to (\x,\d); 
\node at (-1,-2) {\dots};  \node at (4,-2) {\dots}; 
    \end{tikzpicture} \end{array}
\end{gather*}
(at each step we apply induction to the term in the gray box). Now we
are done by part (1) of Lemma \ref{lem:ga}.

Now assume that $k$ is odd. Then we have
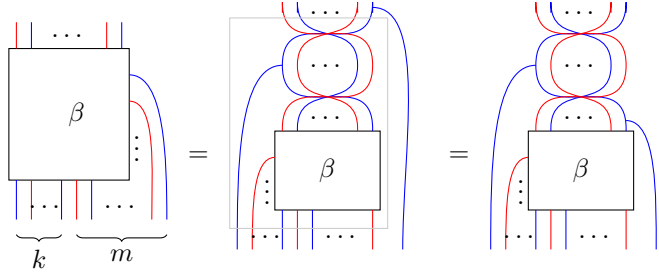
\[
  \begin{array}{c}
    \begin{tikzpicture}[xscale=0.2,yscale=0.35]
      \def\y{5};
     \def\z{-1.5};
\def\c{-3}
 \def\w{-4.5};
\draw (-1.5,0) rectangle (6.5,\y); 
 \foreach\x in {-1,5} \draw[red] (\x,\y) -- (\x,\y+1); 
 \foreach\x in {0,6} \draw[blue] (\x,\y) -- (\x,\y+1); 
\node at (2.5,\y+0.5) {\dots}; 
\draw[blue] (6.5,\y-1) to[out=0,in=90] (9,\z); 
\draw[red] (6.5,\y-2) to[out=0,in=90] (8,\z);
\foreach\x in {-1,2,4} \draw[blue] (\x,0) to (\x,\z); 
\foreach\x in {0,3} \draw[red] (\x,0) to (\x,\z); 
\node at (1,-1) {\dots}; 
\node at (6,-1) {\dots}; 
\node at (7,\y/2-1) {\vdots}; 
\draw [ decoration={brace,mirror,raise=0.2cm},decorate] (-1,\z) --(2,\z);
\draw [ decoration={brace,mirror,raise=0.2cm},decorate] (3,\z) --(9,\z);
\node at (3,\y/2) {$\b$}; \node[below] at (0.5,\z-0.8) {$k$}; \node[below] at (6,\z-0.8) {$m$};
    \end{tikzpicture}
  \end{array}
= 
  \begin{array}{c}\begin{tikzpicture}[xscale=0.2,yscale=0.35]
      \def\y{3};
\def\s{1.5};
     \def\z{-1.5};
\def\c{-3}
 \def\w{-4.5};
\draw (-0.5,0) rectangle (6.5,\y); \node at (3,\y/2) {$\b$};

\foreach \x in {1,6} \draw[blue] (\x,\y) -- (\x,\y+0.1) to[out=90,in=-90] (6-\x,\y+2.5)
to[out=90,in=-90] (\x,\y+4.9) -- (\x,\y+\s+3);
\foreach \x in {0,5} \draw[red] (\x,\y) -- (\x,\y+0.1) to[out=90,in=-90] (6-\x,\y+2.5)
to[out=90,in=-90] (\x,\y+4.9) -- (\x,\y+\s+3);

\foreach \q in {0.5,2.5,4.5} \node at (3,\y+\q) {\dots}; 
\draw[blue] (6,\y+4.7) to[out=0,in=90] (8,\z); 
\draw[red] (-0.5,\y-1) to[out=180,in=90] (-2,\z); 
\draw[blue] (0,\y+\s+1) to[out=180,in=90] (-3,\z); 
\foreach\x in {0,2} \draw[blue] (\x,0) to (\x,\z); 
\foreach\x in {1,6} \draw[red] (\x,0) to (\x,\z); 
\node at (-1,-1) {\dots};  \node at (4,-1) {\dots};  
\node at (-1,\y/2-0.5) {\vdots}; 
\node[below,white] at (3,\z-0.8) {$m$}; 
\draw[gray!40!white] (7,\z+0.8) rectangle (-3.5,\y+\s+2.8) ;
    \end{tikzpicture} \end{array}
=
  \begin{array}{c}\begin{tikzpicture}[xscale=0.2,yscale=0.35]
      \def\y{3};
\def\s{1.5};
     \def\z{-1.5};
\def\c{-3}
 \def\w{-4.5};
\draw (-0.5,0) rectangle (6.5,\y); \node at (3,\y/2) {$\b$};

\foreach \x in {1,6} \draw[blue] (\x,\y) -- (\x,\y+0.1) to[out=90,in=-90] (6-\x,\y+2.5)
to[out=90,in=-90] (\x,\y+4.9) -- (\x,\y+\s+3);
\foreach \x in {0,5} \draw[red] (\x,\y) -- (\x,\y+0.1) to[out=90,in=-90] (6-\x,\y+2.5)
to[out=90,in=-90] (\x,\y+4.9) -- (\x,\y+\s+3);

\foreach \q in {0.5,2.5,4.5} \node at (3,\y+\q) {\dots}; 
\draw[blue] (6,\y+0.4) to[out=0,in=90] (8,\z); 
\draw[red] (-0.5,\y-1) to[out=180,in=90] (-2,\z); 
\draw[blue] (0,\y+\s+1) to[out=180,in=90] (-3,\z); 
\foreach\x in {0,2} \draw[blue] (\x,0) to (\x,\z); 
\foreach\x in {1,6} \draw[red] (\x,0) to (\x,\z); 
\node at (-1,-1) {\dots};  \node at (4,-1) {\dots};  
\node at (-1,\y/2-0.5) {\vdots}; 
\node[below,white] at (3,\z-0.8) {$m$}; 
    \end{tikzpicture} \end{array}
\]
(again for the first step we apply induction and Proposition
\ref{gobblenew} to simplify the term in the gray box).
For the second equality we have used the relation
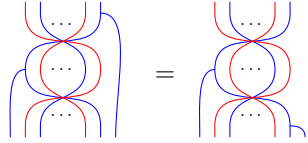
\[
\begin{array}{c}\begin{tikzpicture}[xscale=-0.2, yscale=0.3]
 \foreach\x in{0,4} \draw[blue] (\x,-0.5) to (\x,0) to[out=90,in=-90] (5-\x,2.5) to (5-\x,2.5) to[out=90,in=-90] (\x,5) to  (\x,5.5);
 \foreach\x in{1,5} \draw[red] (\x,-0.5) to  (\x,0) to[out=90,in=-90] (5-\x,2.5) to (5-\x,2.5) to[out=90,in=-90] (\x,5)  to  (\x,5.5);
\draw[blue] (5,2.5) to[out=0,in=90] (6,-0.5);
\draw[blue] (0,5) to[out=180,in=90] (-1,-0.5);
\foreach\y in {0.5,2.5,4.5} \node[scale=0.7] at (2.5,\y) {$\dots$};
\end{tikzpicture}\end{array}
=
\begin{array}{c}\begin{tikzpicture}[xscale=-0.2, yscale=0.3]
 \foreach\x in{0,4} \draw[blue] (\x,-0.5) to (\x,0) to[out=90,in=-90] (5-\x,2.5) to (5-\x,2.5) to[out=90,in=-90] (\x,5) to  (\x,5.5);
 \foreach\x in{1,5} \draw[red] (\x,-0.5) to  (\x,0) to[out=90,in=-90] (5-\x,2.5) to (5-\x,2.5) to[out=90,in=-90] (\x,5)  to  (\x,5.5);
\draw[blue] (5,2.5) to[out=0,in=90] (6,-0.5);
\draw[blue] (0,0) to[out=180,in=90] (-1,-0.5);
\foreach\y in {0.5,2.5,4.5} \node[scale=0.7] at (2.5,\y) {$\dots$};
\end{tikzpicture}\end{array}
\]
which follows by induction (or can be checked direction from two
colour associativity). Now the argument proceeds as in the
case of $k$ even.

\subsection{Proof of Proposition \ref{gobblenew} in general.} \label{proof:gen}

Now take $W$ any Coxeter system. As in the dihedral case it is enough
to establish the following claim:

\begin{lem}  \label{lem:ga}
Let $\un{w} \in Ex(S)$ with $\un{w}_* = x$. In
  $\SD^{\not < x}$ we have:
  \begin{enumerate}
  \item \begin{gather*}
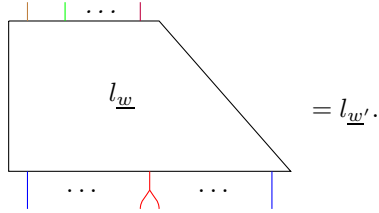
  \begin{array}{c}
    \begin{tikzpicture}[scale=0.5]
      \draw (-0.5,4) -- (3.5,4) -- (7,0) -- (-0.5,0) -- (-0.5,4);
   \node at (2.5,2) {$l_{\un{w}}$};
       \foreach \x in {0} \draw[brown] (\x,4) -- (\x,4.5);
      \foreach \x in {1} \draw[green] (\x,4) -- (\x,4.5);
      \draw[purple] (3,4) -- (3,4.5);
      \node at (2,4.25) {\dots};
       \foreach \x in {0,6.5} \draw[blue] (\x,0) -- (\x,-1);
      \foreach \x in {1.5,5} \node at (\x,-0.5) {$\dots$};
     \coordinate (mid) at (3.25,-0.5);
     \draw[red] (3,-1) to [out=90, in = -120] (mid);
     \draw[red] (3.5,-1) to [out=90, in = -60] (mid);
     \draw[red] (3.25,0) to (mid);
    \end{tikzpicture}
  \end{array} = l_{\un{w}'}.
\end{gather*}
  \item 
  If $\un{w}$ is of the form $\un{w} = u_1 \dots u_n$ with $u_i
  u_{i+1} \dots u_{i+m-1} = s t s \dots $ ($m$ times). Then
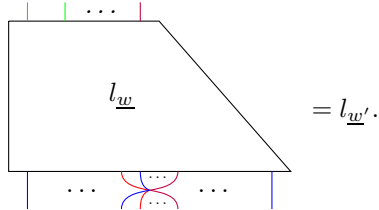
\begin{gather*}
  \begin{array}{c}
    \begin{tikzpicture}[scale=0.5]
      \draw (-0.5,4) -- (3.5,4) -- (7,0) -- (-0.5,0) -- (-0.5,4);
   \node at (2.5,2) {$l_{\un{w}}$};
       \foreach \x in {0} \draw[brown] (\x,4) -- (\x,4.5);
      \foreach \x in {1} \draw[green] (\x,4) -- (\x,4.5);
      \draw[purple] (3,4) -- (3,4.5);
      \node at (2,4.25) {\dots};
       \foreach \x in {0,6.5} \draw[blue] (\x,0) -- (\x,-1);
      \foreach \x in {1.5,5} \node at (\x,-0.5) {$\dots$};
      \coordinate (c) at (3.25,-0.5);
      \def\c{red}; 
      \draw[\c] (2.5,0) to[out=-90,in=170] (c);
      \draw[\c] (3,-1) to[out=90,in=-160] (c);
      \def\c{blue};
      \draw[\c] (2.5,-1) to[out=90,in=-170] (c);
      \draw[\c] (3,0) to[out=-90,in=160] (c);
      \def\c{purple};
      \draw[\c] (4,-1) to[out=90,in=-10] (c);
      \draw[\c] (4,0) to[out=-90,in=10] (c);
      \foreach \y in{-0.15,-0.85} \node[scale=0.6] at (3.5,\y) {$\dots$};
      \def\c{blue};
    \end{tikzpicture}
  \end{array} = l_{\un{w}'}.
\end{gather*}
  \end{enumerate}
(In (1) and (2) the definition of $\un{w}'$ should be clear from the diagrams.)
\end{lem}

As in the dihedral case (1) is immediate from the definitions. It
remains to prove (2). We depict $l_{\un{w}}$ schematically as
follows:
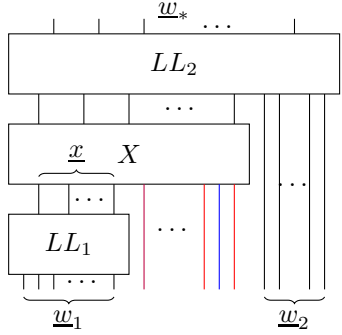
\[
\begin{array}{c} \begin{tikzpicture}[scale=0.4]
\draw (-0.5,0) rectangle (3.5,2); \node at (1.5,1) {$LL_1$}; 
\draw (-0.5,3) rectangle (7.5,5); \node at (3.5,4) {$X$}; 
\draw (-0.5,6) rectangle (10.5,8); \node at (5,7) {$LL_2$}; 
\foreach \x in {0,0.5,1,3} \draw(\x,-0.5) -- (\x,0); 
\node at (2,-0.25) {\dots};
\draw[red] (7,-0.5) -- (7,3);\draw[blue] (6.5,-0.5) --(6.5,3);\draw[red] (6,-0.5) -- (6,3);\draw[purple] (4,-0.5) -- (4,3); 
\node at (5,1.5) {\dots};
\foreach \x in {8,8.5,9.5,10} \draw(\x,6) -- (\x,-0.5); 
\node at (9,3) {\dots};
\foreach \x in {0.5,1.5,3} \draw (\x,2) -- (\x,3); 
\node at (2.25,2.5) {\dots};
\foreach \x in {0.5,2,3.5,7} \draw (\x,5) -- (\x,6); 
\node at (5.25,5.5) {\dots};
\draw [ decoration={brace,mirror,raise=0.1cm},decorate] (0,-0.5) to (3,-0.5);
\draw [ decoration={brace,mirror,raise=0.1cm},decorate] (8,-0.5) --(10,-0.5);
\node at (1.5,-1.5) {$\un{w}_1$};
\node at (9,-1.5) {$\un{w}_2$};
\draw [ decoration={brace,raise=0.1cm},decorate] (0.5,3) to (3,3);
\node at (1.75,4) {$\un{x}$};
\foreach \x in {1,2.5,4,9} \draw(\x,8) -- (\x,8.5); 
\node at (6.5,8.25) {\dots};
\node at (5,8.75) {$\un{w}_*$};
\end{tikzpicture} \end{array}
\]
(Here $\un{w}_1 = u_1 \dots u_{i-1}$, $\un{w}_2 = u_{i+m} \dots
u_n$ and $\un{x}$ is a reduced expression for $(\un{w}_{\le i-1}^{\un{e}_{\le i-1}})_{\bullet}$.) Let us write $(\un{w}_{\le i-1}^{\un{e}_{\le i-1}})_{\bullet}$ as $x' v$ where $x'$ is
  minimal in $x' \langle s, t \rangle$ and $v \in \langle s, t
  \rangle$. By Corollary \ref{cor:ll} we can assume that $\un{x}$ is
  of the form $\un{x'} \hspace{2pt}\un{v}$ where $\un{x}'$ (resp. $\un{v}$)  is a
  reduced expression for $x'$ (resp. $v$).

Moreover, by the construction of light leaves we can assume that $X$ has
the form:
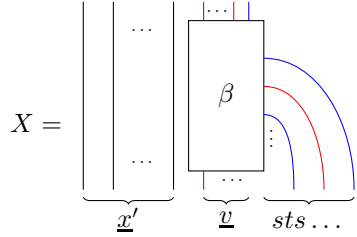
\begin{gather*}
X = 
  \begin{array}{c}
    \begin{tikzpicture}[xscale=0.4,yscale=0.5]
      \def\h{4}
      \def\w{2}
      \foreach \x in {0,1,3} \draw(\x,-0.5) -- (\x,\h+0.5); 
   \foreach \y in {0.25,\h-0.25} \node[scale=0.7] at (2,\y) {\dots};
      \draw (3.5,0) rectangle (4+\w,\h); \node at (3.75+\w/2,\h/2) {$\b$}; 
      \draw[blue] (4+\w-0.5,\h) -- (4+\w-0.5,\h+0.5); 
      \draw[red] (4+\w-1,\h) -- (4+\w-1,\h+0.5);
      \draw[purple] (4+\w-2,\h) -- (4+\w-2,\h+0.5);
      \node[scale=0.7]  at (4+\w-1.5,\h+0.25) {\dots};
      \draw[blue] (4+\w,3) to[out=0,in=90] (4+\w+3,-0.5); 
      \draw[red] (4+\w,2.25) to[out=0,in=90] (4+\w+2,-0.5);
      \draw[blue] (4+\w,1.5) to[out=0,in=90] (4+\w+1,-0.5);
      \node[scale=0.7] at (5,-0.25) {\dots};
      \node[scale=0.7] at (4+\w+0.25,1) {\vdots};
      \draw[purple] (4,0) -- (4,-0.5); 
\draw [ decoration={brace,mirror,raise=0.1cm},decorate] (0,-0.5) to
(3,-0.5); 
\node at (1.5,-1.3) {$\un{x}'$};
\draw [ decoration={brace,mirror,raise=0.1cm},decorate] (4,-0.5) to
(5.5,-0.5); 
\node at (4.75,-1.3) {$\un{v}$};
\draw [ decoration={brace,mirror,raise=0.1cm},decorate] (6,-0.5) to
(9,-0.5); 
\node at (7.5,-1.3) {$sts\dots$};
    \end{tikzpicture}
  \end{array}
\end{gather*}
Now the claim reduces to a rank 2 calculation, which we have covered in the
previous section.

\section{Properties of gobbling morphisms}

The following important property of gobbling morphisms explains
their name:

\begin{prop} \label{dotgobble}
Let $\un{w} = u_1 \dots u_n$. Fix $1 \le i \le n$ and set $\un{w}' :=
u_1 \dots u_{i-1} u_{i+1} \dots u_n$.
  We have
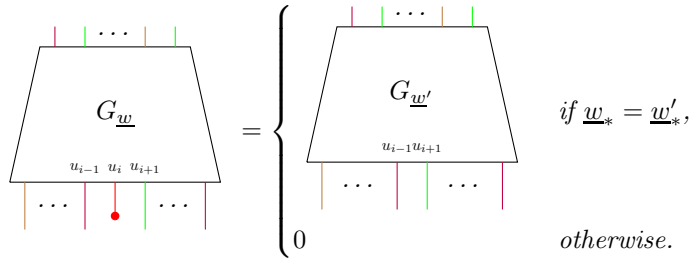
\[
\begin{array}{c}  \begin{tikzpicture}[xscale=0.4,yscale=0.9]
\def\yb{-0.7}
\def\yt{2.3}
\draw (-3.5,0) -- (3.5,0) -- (2.5,2) -- (-2.5,2) -- (-3.5,0);
\node at (0,1) {$G_{\un{w}}$}; 
\draw[brown] (-3,0) -- (-3,\yb); 
\draw[purple] (-1,0) -- (-1,\yb); \draw[green] (1,0) -- (1,\yb);
\draw[purple] (3,0) -- (3,\yb);
\foreach \x in {2,-2} \node at (\x,\yb/2) {\dots}; 
\draw[red] (0,0) -- (0,-0.5); 
 \node[circle,fill,draw,inner
      sep=0mm,minimum size=1mm,color=red] at (0,-0.5) {};
\draw[green] (2,2) -- (2,\yt); 
\draw[purple] (-2,2) -- (-2,\yt);
\draw[green] (-1,2) -- (-1,\yt);
\draw[brown] (1,2) -- (1,\yt);
\node at (0,2.2) {\dots};
\def\s{0.6}
\node[scale=\s] at (-1,0.2) {$u_{i-1}$};
\node[scale=\s] at (0,0.2) {$u_{i}$};
\node[scale=\s] at (1,0.2) {$u_{i+1}$};
\end{tikzpicture} \end{array} 
=
\begin{cases}
\begin{array}{c}  \begin{tikzpicture}[xscale=0.4,yscale=0.9]
\def\yb{-0.7}
\def\yt{2.3}
\draw (-3.5,0) -- (3.5,0) -- (2.5,2) -- (-2.5,2) -- (-3.5,0);
\node at (0,1) {$G_{\un{w}'}$}; 
\draw[brown] (-3,0) -- (-3,\yb); 
\draw[purple] (-0.5,0) -- (-0.5,\yb); \draw[green] (0.5,0) -- (0.5,\yb);
\draw[purple] (3,0) -- (3,\yb);
\foreach \x in {1.75,-1.75} \node at (\x,\yb/2) {\dots}; 
\draw[green] (2,2) -- (2,\yt); 
\draw[purple] (-2,2) -- (-2,\yt);
\draw[green] (-1,2) -- (-1,\yt);
\draw[brown] (1,2) -- (1,\yt);
\node at (0,2.2) {\dots};
\def\s{0.6}
\node[scale=\s] at (-0.5,0.2) {$u_{i-1}$};
\node[scale=\s] at (0.5,0.2) {$u_{i+1}$};
\end{tikzpicture} \end{array} 
& \text{if $\un{w}_* = \un{w}_*'$,} \\
0 & \text{otherwise.}
\end{cases}
\]
\end{prop}

In the proof we will need the following lemma:

\begin{lem} \label{dotpitchforks}
The following relations hold. In each case $P$ is a
  linear combination of diagrams, all of which end in pitchforks:
\begin{gather*}
\text{($m_{st}$ odd)} \quad
\begin{array}{c}\begin{tikzpicture}[xscale=0.2,yscale=0.6]
\coordinate (c) at (3,1);
\draw[blue] (0,0) to[out=90,in=-175] (c); \draw[red] (1,0) to[out=90,in=-170] (c);\draw[red] (5,0) to[out=90,in=-10] (c); \draw[blue] (6,0) to[out=90,in=-5] (c);
\draw[red] (0,2) to[out=-90,in=175] (c); \draw[blue] (1,2) to[out=-90,in=170] (c);\draw[blue] (5,2) to[out=-90,in=10] (c); \draw[red] (6,2) to[out=-90,in=5] (c);
\foreach \y in {0.5,1.5} \node at (3,\y) {$\dots$};
 \node[circle,fill,draw,inner
      sep=0mm,minimum size=1mm,color=blue] at (0,0) {};
\foreach \x in {1,5} \draw[red] (\x,0) to (\x,-0.3);\draw[blue] (6,0) to (6,-0.3);
\end{tikzpicture}\end{array}
=
\begin{array}{c}\begin{tikzpicture}[xscale=0.2,yscale=0.6]
\draw[red] (1,-0.5) to (0,2); \draw[blue] (5,-0.5) to (4,2);
\draw[blue] (2,-0.5) to (1,2);
\node at (3,1) {$\dots$};
\draw[red] (5,2) -- (5,1.5);
 \node[circle,fill,draw,inner
      sep=0mm,minimum size=1mm,color=red] at (5,1.5) {};
\end{tikzpicture}\end{array} + P
\qquad
\text{($m_{st}$ even)} \quad
\begin{array}{c}\begin{tikzpicture}[xscale=0.2,yscale=0.6]
\coordinate (c) at (3,1);
\draw[blue] (0,0) to[out=90,in=-175] (c); \draw[red] (1,0) to[out=90,in=-170] (c);\draw[blue] (5,0) to[out=90,in=-10] (c); \draw[red] (6,0) to[out=90,in=-5] (c);
\draw[red] (0,2) to[out=-90,in=175] (c); \draw[blue] (1,2) to[out=-90,in=170] (c);\draw[red] (5,2) to[out=-90,in=10] (c); \draw[blue] (6,2) to[out=-90,in=5] (c);
\foreach \y in {0.5,1.5} \node at (3,\y) {$\dots$};
 \node[circle,fill,draw,inner
      sep=0mm,minimum size=1mm,color=blue] at (0,0) {};
\foreach \x in {1,6} \draw[red] (\x,0) to (\x,-0.3);\draw[blue] (5,0) to (5,-0.3);
\end{tikzpicture}\end{array}
=
\begin{array}{c}\begin{tikzpicture}[xscale=0.2,yscale=0.6]
\draw[red] (1,-0.5) to (0,2); \draw[red] (5,-0.5) to (4,2);
\draw[blue] (2,-0.5) to (1,2);
\node at (3,1) {$\dots$};
\draw[blue] (5,2) -- (5,1.5);
 \node[circle,fill,draw,inner
      sep=0mm,minimum size=1mm,color=blue] at (5,1.5) {};
\end{tikzpicture}\end{array} + P \\
\text{($m_{st}$ odd)} \quad
\begin{array}{c}\begin{tikzpicture}[xscale=0.2,yscale=0.6]
\coordinate (c) at (3,1);
\draw[blue] (0,0) to[out=90,in=-175] (c); \draw[blue] (6,0) to[out=90,in=-5] (c);
\draw[brown] (2,0) to[out=90,in=-150] (c);\draw[brown] (4,0) to[out=90,in=-40] (c);
\draw[red] (0,2) to[out=-90,in=175] (c); 
\draw[red] (6,2) to[out=-90,in=5] (c);
\draw[brown] (3,1) -- (3,2);
\foreach \y in {0.5,1.5} {
\foreach \x in {1.2,4.8}
\node[scale=0.5] at (\x,\y) {$\dots$};}
\draw[gray] (3,1) -- (3,0);
 \node[circle,fill,draw,inner
      sep=0mm,minimum size=1mm,color=gray] at (3,0) {};
\foreach \x in {2,4} \draw[brown] (\x,0) to (\x,-0.3);\foreach \x in {0,6} \draw[blue] (\x,0) to (\x,-0.3);
\end{tikzpicture}\end{array}
=
\begin{array}{c}\begin{tikzpicture}[xscale=0.2,yscale=0.6]
\coordinate (c) at (3,1);
\draw[brown] (2,0) to[out=90,in=-150] (c);\draw[brown] (4,0) to[out=90,in=-40] (c);
\draw[brown] (3,1) -- (3,2);
\foreach \y in {0.5} {
\foreach \x in {1.2,4.8}
\node[scale=0.5] at (\x,\y) {$\dots$};}
\draw[blue] (0,0) to[out=90,in=-90] (1,2);\draw[blue] (6,0) to[out=90,in=-90] (5,2);
\foreach \x in {2,4} \draw[brown] (\x,0) to (\x,-0.3);\foreach \x in {0,6} \draw[blue] (\x,0) to (\x,-0.3);
\foreach \x in {0,6} {
\draw[red] (\x,2) -- (\x,1.5);
 \node[circle,fill,draw,inner
      sep=0mm,minimum size=1mm,color=red] at (\x,1.5) {}; }
\end{tikzpicture}\end{array} + P
\qquad
\text{($m_{st}$ even)} \quad
\begin{array}{c}\begin{tikzpicture}[xscale=0.2,yscale=0.6]
\coordinate (c) at (3,1);
\draw[blue] (0,0) to[out=90,in=-175] (c); \draw[red] (6,0) to[out=90,in=-5] (c);
\draw[brown] (2,0) to[out=90,in=-150] (c);\draw[brown] (4,0) to[out=90,in=-40] (c);
\draw[red] (0,2) to[out=-90,in=175] (c);
\draw[blue] (6,2) to[out=-90,in=5] (c);
\draw[brown] (3,1) -- (3,2);
\foreach \y in {0.5,1.5} {
\foreach \x in {1.2,4.8}
\node[scale=0.5] at (\x,\y) {$\dots$};}
\draw[gray] (3,1) -- (3,0);
 \node[circle,fill,draw,inner
      sep=0mm,minimum size=1mm,color=gray] at (3,0) {};
\foreach \x in {2,4} \draw[brown] (\x,0) to (\x,-0.3);\foreach \x in {0} \draw[blue] (\x,0) to (\x,-0.3);\foreach \x in {6} \draw[red] (\x,0) to (\x,-0.3);
\end{tikzpicture}\end{array}
=
\begin{array}{c}\begin{tikzpicture}[xscale=0.2,yscale=0.6]
\coordinate (c) at (3,1);
\draw[brown] (2,0) to[out=90,in=-150] (c);\draw[brown] (4,0) to[out=90,in=-40] (c);
\draw[brown] (3,1) -- (3,2);
\foreach \y in {0.5} {
\foreach \x in {1.2,4.8}
\node[scale=0.5] at (\x,\y) {$\dots$};}
\draw[blue] (0,0) to[out=90,in=-90] (1,2);\draw[red] (6,0) to[out=90,in=-90] (5,2);
\foreach \x in {2,4} \draw[brown] (\x,0) to (\x,-0.3);\foreach \x in {0} \draw[blue] (\x,0) to (\x,-0.3);\foreach \x in {6} \draw[red] (\x,0) to (\x,-0.3);
\foreach \x/\c in {0} {
\draw[red] (\x,2) -- (\x,1.5);
 \node[circle,fill,draw,inner
      sep=0mm,minimum size=1mm,color=red] at (\x,1.5) {}; }
\foreach \x/\c in {6} {
\draw[blue] (\x,2) -- (\x,1.5);
 \node[circle,fill,draw,inner
      sep=0mm,minimum size=1mm,color=blue] at (\x,1.5) {}; }
\end{tikzpicture}\end{array}+ P
\end{gather*}
(In the second pair of relations the brown and gray lines can be
chosen to be either red and blue, or blue and red.)
\end{lem}

\begin{proof}
  This is a consequence of Elias' Jones-Wenzl relation. (The only
  tricky point is to check that the coefficient of 1 in the right
  hand term is correct. The reader can verify directly that this is the case
  for $m_{st} = 2, 3, 4$ using the relations we have given.)
\end{proof}

\begin{proof}[Proof of Proposition \ref{dotgobble}]
  If $\un{w}'_* \ne \un{w}_*$ then $\un{w}'_* < \un{w}_*$ because
  $\un{w}'$ is a subexpression of $\un{w}$ and the result is zero
  because we are working in   $\SD^{\not < \un{w}_*}$.

So now we assume that $\un{w}_*' = \un{w}_*$. Let us induct on the
number $N$ of $2m_{st}$-valent vertices which occur in
$G_{\un{w}}$. If $N =0$ then $G_{\un{w}}$ consists only of trivalent
vertices and the result is clear by the Frobenius unit relation \S\ref{2.7.1}.

Now we examine what happens when we apply the relations to 
simplify the diagram on the left hand side of the proposition. If the
dot first meets a trivalent vertex then the result is clear from the
Frobenius unit relation \S\ref{2.7.1} (the source and target have the same star product
because $s*s = s$). If the dot meets a $2m$-valent vertex then we can
apply Lemma \ref{dotpitchforks} to expand it (and remove the $2m$-valent vertex
in question). By Lemma \ref{lem:pitchfork} we can ignore all the
terms labelled $P$ in the lemma. Now we are done by induction (note
that in each of the relations in Lemma \ref{dotpitchforks} the
expression corresponding to through strands on the bottom
and top boundary of the pictured right hand term have the same
$*$-product).
\end{proof}

\section{Morphisms without D1's}

\subsection{The nil Hecke ring}

Denote by $Q$ the field of fractions of $R$.

Consider the smash product $Q_W := Q * W$. That is, $Q * W$ is a free left
$Q$-module with basis $\{\delta_w \; | \; w \in W \}$ and multiplication given by
\[
(f \delta_x )(g \delta_y) = f(xg) \delta_{xy}.
\]
Inside $Q_W$ we consider the elements
\[
D_s := \frac{1}{\alpha_s} (\delta_{\id} - \delta_s) = (\delta_{\id} +
\delta_s) \frac{1}{\alpha_s}.
\]
The elements $D_s$ satisfy the relations:
\begin{align}
D_s^2 &= 0; \label{eq:D^2} \\
D_sD_t \dots &= D_tD_s \dots \quad \text{($m_{st}$-factors on both
  sides)} \label{eq:braid}\\
D_sf &= (sf)D_s + \partial_s(f) \quad \text{for all $f \in Q$.}\label{eq:pol}
\end{align}
If $y \in W$ and $\un{y} = st\dots u$ is a reduced
expression then, by \eqref{eq:braid},  we obtain well-defined elements
\[
D_y := D_sD_t \dots D_u  \in Q_W.
\]
The \emph{nil Hecke ring} $NH$ is defined to be the $R$-subring of $Q * W$
generated by $\{ D_w \; | \; w \in W\}$. (This is not the definition
of Kostant-Kumar \cite{KK}, but
agrees with it for the realizations they consider.) As a left
$R$-module $NH$ is free with basis $\{ D_w \; | \; w \in W \}$.

\subsection{} Given a word $\un{w} = (s_1, \dots, s_m)$ and a subexpression $\un{e}$ we have a
corresponding light leaves map $LL_{\un{e}}: \un{w} \to (\un{w}^{\un{e}})_\bullet$ in
  $\SD^{\not < \un{w}^{\un{e}}}$. Now it is clear from construction
  that if the decoration of $\un{e}$ has no D1's then $LL_{\un{e}}$
  has the form
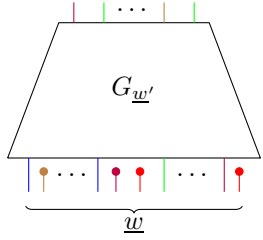
\[
 \begin{tikzpicture}[xscale=0.4,yscale=0.9]
\def\yb{-0.5}
\def\yt{2.3}
\draw (-4.2,0) -- (4.2,0) -- (2.5,2) -- (-2.5,2) -- (-4.2,0);
\node at (0,1) {$G_{\un{w}'}$}; 
\draw[brown] (-3,\yb) -- (-3,-0.2); 
 \node[circle,fill,draw,inner
      sep=0mm,minimum size=1mm,color=brown] at (-3,-0.2) {};
\draw[blue] (-3.5,0) -- (-3.5,\yb); 
\draw[purple] (-0.6,\yb) -- (-0.6,-0.2);
 \node[circle,fill,draw,inner
      sep=0mm,minimum size=1mm,color=purple] at (-0.6,-0.2) {};
\draw[green] (1,0) -- (1,\yb);
\draw[blue] (-1.2,0) -- (-1.2,\yb);
\draw[purple] (3,0) -- (3,\yb);
\draw[red] (3.5,\yb) -- (3.5,-0.2);
 \node[circle,fill,draw,inner
      sep=0mm,minimum size=1mm,color=red] at (3.5,-0.2) {};
\foreach \x in {2,-2} \node at (\x,\yb/2) {\dots}; 
\draw[red] (0.2,\yb) to (0.2, -0.2);
 \node[circle,fill,draw,inner
      sep=0mm,minimum size=1mm,color=red] at (0.2,-0.2) {};
\draw[green] (2,2) -- (2,\yt); 
\draw[purple] (-2,2) -- (-2,\yt);
\draw[green] (-1,2) -- (-1,\yt);
\draw[brown] (1,2) -- (1,\yt);
\node at (0,2.2) {\dots};
\def\s{0.6}

\draw [ decoration={brace,mirror,raise=0.1cm},decorate] (-3.6,-0.6) to (3.6,-0.6);
\node at (0,-1) {$\un{w}$};
\end{tikzpicture} 
\]
where $\un{w}'$ is the subword of $\un{w}$ consisting of all $s_i$
such that $d_ie_i \ne U0$ and $G_{\un{w}'}$ denotes the gobbling
morphism associated to $\un{w}'$. In particular, the
light leaves morphism $LL_{\un{e}}$ is canonical in this case.

Consider two subexpressions $\un{e^1}$ and $\un{e^2}$ of $\un{w}$
such that $(\un{w}^{\un{e^1}})_\bullet = (\un{w}^{\un{e^2}})_\bullet$. Define an
element of the nil Hecke ring as the product
\[
f(\un{e^1}, \un{e^2}) = f_1 f_2 \dots f_m
\]
where
\[
f_i = \begin{cases} \alpha_{s_i} & \text{if $e_i^1 = e_i^2 =
    U0$,}\\
    1 & \text{if exactly one of }e_i^1 \text{ and } e_i^2 \text{ is } U0, \\
D_{s_i} &  \text{otherwise}. \end{cases}
\]
Finally, define $d(\un{e^1}, \un{e^2})$ as the coefficient of
$D_{(\un{w}^{\un{e^1}})_\bullet }$ in $f(\un{e^1}, \un{e^2})$. Hence
$d(\un{e^1}, \un{e^2}) \in R$.

\begin{thm}\label{thm:main} If $(\un{w}^{\un{e^1}})_\bullet =
  (\un{w}^{\un{e^2}})_\bullet$ and the decorations of $\un{e^1}$ and $\un{e^2}$ have no $D1$, then 
  \[
\langle LL_{\un{e^1}}, LL_{\un{e^2}} \rangle = d(\un{e^1}, \un{e^2}).
\]
\end{thm}

\begin{remark}
Given a subexpression $\un{e}$ of $\un{w}$ let $J := \{ 1 \le i \le m
\; | \; e_i = 1 \}$ and write $J = \{ i_1, \dots, i_r \}$ with $i_1 <
i_2 < \dots < i_r$. Then $\un{e}$ has no $D1$'s if and only if
$(s_{i_1}, s_{i_2}, \dots, s_{i_r})$ is reduced. That is, $\un{e}$ has
no $D1$s if the expression obtained from $\un{w}^{\un{e}} =
(s_1^{e_1}, \dots, s_m^{e_m})$ by
  deleting all occurrences of $\id$ is reduced.
\end{remark}

\begin{proof}
Set $x=(\un{w}^{\un{e^1}})_\bullet=(\un{w}^{\un{e^2}})_\bullet$. Let $\un{w'}^1$ be the expression obtained from $\un{w}$ by deleting all the places $i$ with $d_i e_i^1=U0$ and $\un{w'}^2$ be the expression obtained from $\un{w}$ by deleting all the places $i$ with $d_i e_i^2=U0$. Then $x=(\un{w'}^1)_*=(\un{w'}^2)_*$. Therefore:
\[
LL_{\un{w}, \un{e^1}} \circ \overline{LL}_{\un{w}, \un{e^2}}=
G_{\un{w'}^1} \circ K \circ \overline{G}_{\un{w'}^2}= 
  \begin{array}{c}
    \begin{tikzpicture}[scale=0.5]
      \draw (0,1) -- (2,3) -- (4,3) -- (6,1) -- (0,1);
      \draw (0,-1) -- (2,-3) -- (4,-3) -- (6,-1) -- (0,-1);
      \node at (3,2) {$G_{\un{w'}^1}$};\node at (3,-2) {$\overline{G}_{\un{w'}^2}$};
      \draw (0,1) -- (6,1) -- (6,-1) -- (0,-1) -- (0,1);
      \node at (6.7,1) {$\un{w'}^1$};\node at (6.7,-1) {$\un{w'}^2$};\node
      at (7,0) {$\un{w}$};
\draw[gray] (0,0) -- (2,0);
\draw[gray] (6,0) -- (4,0);
\node at (3,0) {$K$};
    \end{tikzpicture}
  \end{array}
\]The graph $K$ consists of vertical
edges, vertical dotted edges and boxes. For example, 
\begin{gather*}
K =
\begin{array}{c}
  \begin{tikzpicture}[scale=0.8]
\def\l{0.2}
    \foreach \x in {1,5,8,9} \draw (\x,0.5) -- (\x,-0.5); 
    \foreach \x in {2,3,10} 
       { \draw (\x,0.5) -- (\x,-\l);
         \draw[fill] (\x,-\l) circle (2pt); }  
    \foreach \x in {7,12} 
       { \draw (\x,-0.5) -- (\x,\l);
         \draw[fill] (\x,\l) circle (2pt); }  
 \def\l2{0.3}
    \foreach \x in {4,6,11} 
       { 
\draw (\x-\l2,-1*\l2) rectangle (\x + \l2,\l2);
\node at (\x,0) {\scalebox{0.8}{$\alpha_{s_{\x}}$}}; }
  \end{tikzpicture}
\end{array}
\end{gather*}
The graph $K$ is determined by $\un{e^1}$ and $\un{e^2}$ as follows. 
At the $i$th place, we
associate a box labelled by $\a_{s_i}$ if $d_i e_i^1=d_i e_i^2=U0$, we
associate a vertical dotted edge with label $s_i$ if exactly one of
$d_i e_i^1$ and $d_i e_i^2$ is $U0$ and associate a vertical edge 
colored by $s_i$ if neither $d_i e_i^1$ nor $d_i e_i^2$ is $U0$ (here
$* \in \{ U1, D0 \}$):
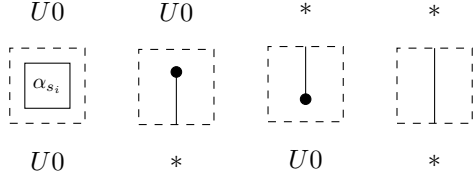
\begin{gather*}
  \begin{array}{c}
    \begin{tikzpicture}
\draw[dashed] (-0.5,-0.5) rectangle (0.5,0.5);
\def\l{0.3}
\draw (-\l,-\l) rectangle (\l,\l);
\node at (0,0) {\scalebox{0.8}{$\alpha_{s_i}$}};
\node at (0,1) {$U0$};\node at (0,-1) {$U0$};
    \end{tikzpicture}
  \end{array} \quad 
  \begin{array}{c}
    \begin{tikzpicture}
\draw[dashed] (-0.5,-0.5) rectangle (0.5,0.5);
\def\l{0.2}
\draw (0,-0.5) to (0,\l);
\draw[fill] (0,\l) circle (2pt);
\node at (0,1) {$U0$};\node at (0,-1) {$*$};
    \end{tikzpicture}
  \end{array} \quad 
  \begin{array}{c}
    \begin{tikzpicture}
\def\l{0.2}
\draw[dashed] (-0.5,-0.5) rectangle (0.5,0.5);
\draw (0,0.5) to (0,-\l);
\draw[fill] (0,-\l) circle (2pt);
\node at (0,1) {$*$};\node at (0,-1) {$U0$};
    \end{tikzpicture}
  \end{array} \quad 
  \begin{array}{c}
    \begin{tikzpicture}
\draw[dashed] (-0.5,-0.5) rectangle (0.5,0.5);
\draw (0,-0.5) to (0,0.5);
\node at (0,1) {$*$};\node at (0,-1) {$*$};
    \end{tikzpicture}
  \end{array}
\end{gather*}

In the category $\SD$, by using the nil Hecke relation
we may write $K$ as an $R$-linear combination of graphs consisting of
vertical edges, dotted edges and broken vertical edges. (See
\eqref{eq:nilHecke} for an example of this process and a broken edge.)
More precisely, 
consider the following sets:
\begin{gather*}
  \un{\e} := \{ i \; | \; (d_ie_i^1, d_ie_i^2) \ne (U0, U0) \},\\
\un{\e}_\phi := \{ i \; | \; \text{exactly one of $d_ie_i^1,
  d_ie_i^2$ is $U0$} \}.
\end{gather*}
Thus $\un{\e}$ indexes  those positions in $K$ which do not
correspond to boxes, and $\un{\e}_\phi$ corresponds to those edges in $K$
which carry one dot.

For any $\un{\g}$ with $\un{\e}_\phi \subset \un{\g} \subset \un{\e}$, let $K_{\un{\g}}$
be the graph obtained from $K$ by removing all boxes and breaking any edge corresponding to
$i \in \un{\g} - \un{\e}_\phi$. For example, with $K$ as
above:
\begin{gather*}
K_{\un{\e}_\phi} =
\begin{array}{c}
  \begin{tikzpicture}[scale=0.8]
\def\l{-0.2} \def\b{0.2}
    \foreach \x in {1,4,6,7}
     {\draw (\x,0.5) -- (\x,-0.5);}     
    \foreach \x in {2,3,8} 
       { \draw (\x,0.5) -- (\x,-\l);
         \draw[fill] (\x,-\l) circle (2pt); }  
    \foreach \x in {5,9} 
       { \draw (\x,-0.5) -- (\x,\l);
         \draw[fill] (\x,\l) circle (2pt); }  
 \def\l2{0.3}
  \end{tikzpicture}
\end{array} \\
K_{\un{\e}} =
\begin{array}{c}
  \begin{tikzpicture}[scale=0.8]
\def\l{-0.2} \def\b{0.2}
    \foreach \x in {1,4,6,7}
     {\draw (\x,0.5) -- (\x,\b); \draw (\x,-\b) -- (\x,-0.5);
     \draw[fill] (\x,\b) circle (2pt); \draw[fill] (\x,-\b) circle (2pt);}
    \foreach \x in {2,3,8} 
       { \draw (\x,0.5) -- (\x,-\l);
         \draw[fill] (\x,-\l) circle (2pt); }  
    \foreach \x in {5,9} 
       { \draw (\x,-0.5) -- (\x,\l);
         \draw[fill] (\x,\l) circle (2pt); }  
 \def\l2{0.3}
  \end{tikzpicture}
\end{array}
\end{gather*}

On the other hand, let $M$ be the $R$-algebra with generators $A_s$
and $B_s$ for $s \in S$ and ``nil Hecke'' relation
\begin{equation}
A_s f=s(f) A_s+\partial_{s}(f)
B_s\label{eq:nhab}
\end{equation}
for $s \in S$ and $f \in R$. Define $F=F_1 \cdots F_m \in M$, where \[
F_i = \begin{cases} \alpha_{s_i} & \text{if $e_i^1 = e_i^2 =
    U0$,}\\
    B_{s_i} & \text{if exactly one of }e_i^1 \text{ and } e_i^2 \text{ is } U0, \\
A_{s_i} &  \text{otherwise}. \end{cases}
\] 
For any $\un{\g}$ with $\un{\e}_\phi \subset \un{\g} \subset \un{\e}$, we define $F_{\un{\g}}=F_{1, \un{\g}} \cdots F_{m, \un{\g}}$, where 
\[
F_{i,\un{\g}} = \begin{cases} B_{s_i} & \text{if } i \in \un{\g}, \\
    A_{s_i} & \text{if } i \in \un{\e}-\un{\g}, \\
1 &  \text{otherwise}. \end{cases}
\]
Using \eqref{eq:nhab} we may write
\begin{equation}
F=\sum_{\un{\e}_\phi \subset \un{\g} \subset \un{\e}}a_{\un{\g}} F_{\un{\g}} 
\quad \text{with $a_{\un{\g}} \in R$}.
\label{eq:abexpansion}
\end{equation}

Because the nil Hecke relation holds in $\SD$, we have $$G_{\un{w'}^1}
\circ K \circ \overline{G}_{\un{w'}^2}=\sum_{\un{\e}_\phi \subset
  \un{\g} \subset \un{\e}} a_{\un{\g}} G_{\un{w'}^1} \circ K_{\un{\g}}
\circ \overline{G}_{\un{w'}^2}.$$

Let $K_{\un{\g}}'$ be the subgraph of $K_{\un{\g}}$ consisting of only
unbroken vertical edges. If $\un{w}'_{\un{\g}}$ denotes the subexpression of
$\un{w}$ corresponding to those $i \in \un{\e} - \un{\g}$, then
$K_{\un{\g}}'$ is the identity on $B_{\un{w}'_\g}$.
By Proposition \ref{dotgobble}, 
\[
G_{\un{w'}^1} \circ K_{\un{\g}} \circ \overline{G}_{\un{w'}^2}=\begin{cases}
G_{\un{w'}_{\un{\g}}} \circ K_{\un{\g}}' \circ \overline{G}_{\un{w'}_{\un{\g}}} & \text{if } (\un{w'}_{\un{\g}})_*=x, \\
0 & \text{otherwise}.
\end{cases}
\]
On the other hand we claim:
\[
G_{\un{w'}_{\un{\g}}} \circ K_{\un{\g}}' \circ
\overline{G}_{\un{w'}_{\un{\g}}} = \begin{cases}
\id_x & \text{if $\un{w}_{\un{\g}}'$ is reduced,} \\
0 & \text{otherwise.} \end{cases}
\]
Indeed, if $\un{w}_{\un{\g}}'$ is reduced then
$G_{\un{w'}_{\un{\g}}}$, $K_{\un{\g}}'$ and
$\overline{G}_{\un{w'}_{\un{\g}}}$ are all canonical isomorphisms
between different representatives for $x$, and if $\un{w}_{\un{\g}}'$
is not reduced then $G_{\un{w'}_{\un{\g}}} \circ K_{\un{\g}}' \circ
\overline{G}_{\un{w'}_{\un{\g}}}$ is an endomorphism of $x$ of
negative degree, and $\End_{\SD^{\ge x}}(x) = R$ is zero in negative
degree. Thus 
$$
\langle LL_{\un{e^1}}, LL_{\un{e^2}} \rangle \id_x = 
G_{\un{w'}^1} \circ K \circ
\overline{G}_{\un{w'}^2}=\sum_{\un{\e}_\phi \subset \un{\g} \subset
  \un{\e}, \atop \un{w}_{\un{\g}}' \text{is a reduced expression for
    $x$.}} a_{\un{\g}} \id_x.$$

We now explain how the right hand side can be computed in the nil
Hecke ring. Consider the homomorphism of $Q$-algebras 
\[p: M \to NH, \quad A_s \mapsto D_s,B_s \mapsto 1.\]
If $\un{\e}_\phi \subset \un{\g} \subset
  \un{\e}$ then $\un{w}_{\un{\g}}'$ is a subexpression of both $\un{w}'^1$
  and $\un{w}'^2$. In particular $(\un{w}_\g')_* \le \un{w}'^1_{*} =
  \un{w}'^2_* = x$ and hence
\[
p(F_{\un{\g}}) \in NH_{\le x} := \bigoplus_{y \le x} RD_y.
\]
Furthermore, in the quotient $NH_{\le x}/NH_{< x}$ (
$NH_{< x} := \bigoplus_{y < x} RD_y$) we have
\begin{equation}
  \label{eq:1}
  p(F_{\un{\g}}) =
  \begin{cases}
    D_x & \text{if $\un{w}_{\un{\g}}'$ is a reduced expression for $x$,}\\
0 & \text{otherwise}.
  \end{cases}
\end{equation}
Hence applying $p$ to \eqref{eq:abexpansion} we obtain
\[
p(F) = \sum_{\un{\e}_\phi \subset \un{\g} \subset
  \un{\e}, \atop \un{w}_{\un{\g}}' \text{is a reduced expression for
    $x$.}} a_{\un{\g}} D_x + NH_{< x}
= \langle LL_{\un{e^1}}, LL_{\un{e^2}} \rangle D_x + NH_{<x}.
\]
From the definitions we see $p(F) = f(\un{e}^1, \un{e}^2)$ and the result follows.
\end{proof}

\section{Examples}

\subsection{}
This example comes from \cite[Example 8.3.1]{KS}. 

Let $W=S_8$. Let 
\begin{gather*} \un{w}=(s_1, s_3, s_2, s_4, s_3, s_5, s_4, s_3, s_2, s_1, s_6, s_7, s_6, s_5, s_4, s_3), \\\un{e}=(U1, U1, U0, U1, U1, U1, U1, U1, U0, D0, U0, U1, U0, D0, D0, D0).
\end{gather*} 
Then $$f(\un{e}, \un{e})=D_1 D_3 \a_2 D_4 D_3 D_5 D_4 D_3 \a_2 D_1 \a_6 D_7 \a_6 D_5 D_4 D_3=2 D_x,$$ where $x=s_1 s_3 s_4 s_3 s_5 s_4 s_3 s_7$.

Notice that $\un{e}$ is the unique $01$-sequence of defect 0 with
$(\un{w}^{\un{e}})_\bullet=x$. Also note that no $D1$ appears in
$\un{e}$. Thus by Theorem \ref{thm:main}, $\langle LL_{\un{w},
  \un{e}}, LL_{\un{w}, \un{e}} \rangle=2$. This tells that the
character of $b_w$ in the Hecke algebra is not that predicted by Kazhdan-Lusztig theory if
the characteristic of $\Bbbk$ is 2.

In fact, the reducibility of the characteristic cycle shown in \cite[Example 8.3.1]{KS} is implied by the
above calculation, using the results of \cite{VW}.

\subsection{}
The following two examples were discovered by Braden
  \cite[Appendix A]{W1}.

\subsubsection{}
Let $W=S_8$. Let 
\begin{gather*}
  \un{w}=(s_3,s_2,s_1,s_5,s_4,s_3,s_2,s_6,s_5,s_4,s_3,s_7,s_6,s_5) \\
\un{e}=(U1, U1,U0, U1, U0, U1, D0, U1, U1, U0, D0, U0, D0, D0)
\end{gather*}
Then $$f(\un{e}, \un{e})=
D_3D_2\a_1D_5\a_4D_3D_2D_6D_5\a_4D_3\a_7D_6D_5 =2 D_x,$$ where
$x=s_2s_3s_2s_5s_6s_5$. Again $\un{e}$ is the unique defect zero
subexpression of $\un{w}$ such that $(\un{w}^{\un{e}})_\bullet = x$.

\subsubsection{}
We repeat the example in $D_4$ considered in \S \ref{ifeg}, this time
using our formula.
Let $W = D_4$ with $S = \{ s, t, u, v\}$ and such that $su = us, sv =
vs, uv = vu$. Let
\begin{gather}
\un{w} = (s,u,v,t,s,u,v),
\end{gather}
Then there are three subexpressions of $\un{w}$ for $x = suv$ of defect zero:
\begin{gather*}
\un{e}^1 = (U0, U1, U1, U0, U1, D0, D0)\\
\un{e}^2 = (U1, U0, U1, U0, D0, U1, D0)\\
\un{e}^3 = (U1, U1, U0, U0, D0, D0, U1)
\end{gather*}
we have
\begin{gather*}
  f(\un{e}^1,\un{e}^1) = \a_s D_uD_v \a_t D_sD_uD_v = 0\\
 f(\un{e}^1,\un{e}^2) = D_v \a_t D_sD_uD_v = -D_x.
\end{gather*}
By symmetry, the matrix $(\langle LL_{e^i}, LL_{e^j}
\rangle)_{1 \le i, j \le 3}$ is given by
\[
\left ( \begin{matrix} 0 & -1 & -1 \\ -1 & 0 & -1 \\ -1 & -1 &
    0 \end{matrix} \right ).
\]

\subsection{}
This example was discovered by the second author \cite{W3}.

Let $W = S_{12} = \langle s_1, s_2, \dots, s_9, s_a, s_b
\rangle$. Consider $\un{w}$ and $\un{e}$ as follows:
\begin{gather*}
(s_1,s_2,s_1,s_3,s_2,s_1,s_5,s_4,s_6,s_5,s_4,s_3,s_7,s_6,s_5,s_4,s_3,s_8,s_7,s_9,s_8,s_7,s_6,s_5,s_a,s_
b,s_a,s_9,s_8,s_7)\\
\scalebox{0.85}{
(U1 U1 U1 U1 U1 U1 U1 U0 U1 U1 U1 U1 U1 U1 U1 U1 U1 U1 U1 U1 U1 U1 U0
D0 U0 U1 U0 D0 D0 D0)}
\end{gather*}
Then $\un{e}$ is the unique defect zero subexpression for
\[
x = s_1s_2s_1s_3s_2s_1s_5s_6s_5s_4s_3s_7s_6s_5s_4s_3s_8s_7s_9s_8s_7s_{b}.
\]
Then $$f(\un{e}, \un{e})=
D_{1213215}\a_4 D_{65437654387987} \a_6 D_5 \a_a D_b \a_a D_{987} =2 D_x.$$

The significance of this example is that it is probably the first
example in type $A$ where $\un{w}_\bullet$ and $x$ lie in the same
two-sided cell.




\end{document}